\documentclass{article}
\usepackage{hyperref}
\usepackage{amssymb}
\usepackage{amsthm}
\usepackage{amsmath}
\usepackage{mathtools}
\usepackage[normalem]{ulem}
\newtheorem{theorem}{Theorem}[section]
\newtheorem{conjecture}[theorem]{Conjecture}
\newtheorem{lemma}[theorem]{Lemma}
\newtheorem{corollary}[theorem]{Corollary}
\newtheorem{proposition}[theorem]{Proposition}
\newtheorem{definition}[theorem]{Definition}
\let\originalmiddle=\middle
\def\middle#1{\mathrel{}\originalmiddle#1\mathrel{}}
\def\ast#1{{^*#1}}
\DeclareMathOperator{\res}{res}
\DeclareMathOperator{\Frac}{Frac}
\DeclareMathOperator{\lcm}{lcm}

\title{The reals as a subset of an ultra­product of finite fields}
\author{Roee Sinai}
\date{May 2026}

\begin{document}

\maketitle


\setcounter{section}{-1}
\begin{abstract}
    In this paper we present new ways to construct external subsets of nonstandard models of arithmetic using mostly internal sets, and show that if an ultra­product of prime finite fields includes a copy of the algebraic real numbers then either this copy or its algebraic closure can be constructed in some of these ways. We also show that no copy of the field of real numbers inside such an ultra­product can ever be constructed in any of these ways, but there is either a hyper­real field or an algebraically closed field of cardinality larger or equal to the continuum that can be.
\end{abstract}
\section{Introduction}
It follows from \cite[Theorem 3.10]{Chatzidakis}, \cite[Theorem 6.5.3]{Hodges}, \cite[Corollary 4.3.13]{CK} and \cite[Proposition 7]{Garner} that there is a field of the form \(\tilde{\mathbb F}=\prod\mathbb F_p/\mathcal F\) where all \(p\) are prime numbers such that \(\mathbb R\cong\mathbb R'\subset\tilde{\mathbb F}\) but \(\tilde{\mathbb F}\) does not contain a square root of \(-1\). A more direct proof can be found in \cite{MSX}, where it is also shown that it is sufficient to use each \(\mathbb F_p\) only once. However, \(\mathbb R'\), or any other subfield of this ultra­product, cannot be internal in the nonstandard model of arithmetic \(\prod\mathbb N/\mathcal F\), because that would mean that \(\mathbb{R}'=\prod_{p}S_p/\mathcal{F}\), where \(S_p\subseteq\mathbb{F}_p\) is a field for a large set of the primes, but no field \(\mathbb{F}_p\) has a subset that is also a field. Additionally, all of these proofs do not only show that \(\mathbb R\) can be embedded in this ultra­product, but also that a hyper­real field can be embedded there, and this hyper­real field contains \(2^\mathfrak c\) different copies of the real numbers, with no copy having any apparent structure or preference over any other. One might ask if this is always the case or if there are other ultra­products of finite fields where there are fewer copies of the real numbers or a copy that can be constructed using mostly internal sets. We show that the former is true, for 3 different definitions of ``constructed using mostly internal sets".

Specifically, we show that if the real algebraic integers can be embedded in \(\tilde{\mathbb F}=\prod\mathbb F_p/\mathcal F\) there is a downwards-closed set \(C\) of \(\ast{\mathbb N}=\prod\mathbb N/\mathcal F\) and an internal function \(f:\tilde{\mathbb F}\to\ast{\mathbb N}\) such that \(f^{-1}[C]\) is either an \(\aleph_1\)-saturated real-closed field or an algebraically-closed field of cardinality at least \(\mathfrak c\), and we conclude in both cases that \(\mathbb R\) can be embedded in such a field. We also prove that this \(f^{-1}[C]\) is an intersection of countably many internal subsets of \(\tilde{\mathbb F}\). We also show that is another pair of function and downwards-closed set such that \(f^{-1}[C]\) is a real-closed field that or a algebraically-closed field, but for this pair \(f^{-1}[C]\) is the union of countably many internal subsets of \(\tilde{\mathbb F}\). For this pair as well we show that \(f^{-1}[C]\) contains a copy of \(\mathbb R\). We also show that no copy of \(\mathbb R\) inside \(\tilde{\mathbb F}\) can be of the form \(f^{-1}[C]\) for a downwards-closed set \(C\) and an internal function \(f\), nor can one be a union or an intersection of countably many internal subsets of \(\tilde{\mathbb F}\).

Following \cite[Section 4.4]{CK}, given a set \(X\), define the sets \(V_i(X)\) for \(i\in\omega+1\) like so: \(V_0(X)=X\), \(V_{n+1}(X)=V_n(X)\cup\mathcal P(V_n(X))\) for \(n\in\omega\) and \[V(X)=V_\omega(X)=\bigcup_{n\in\omega}V_n(X).\] If \(X\) is infinite, given a countably incomplete ultra­filter \(\mathcal F\) on a set \(I\) we construct a nonstandard model \((V(X),V(Y),*)\) for the ultra­product \(Y=\prod_IX/\mathcal F\) and \(*\) the elementary embedding of \((V(X),\in)\) to \((V(Y),\in)\) that sends each element of \(X\) to the corresponding constant function, each set \(A\subseteq X\) to \(\{[y_i]\in Y\mid \{i\in I\mid y_i\in A\}\in\mathcal F\}\) (and in particular \(X\) to \(Y\)) etc. An element of \(V(Y)\) is called ``internal" if it is an element of \(\ast{V_n(X)}\) for some \(n\), or equivalently an element of \(\ast B\) for some \(B\in V(X)\setminus X\). The set of all internal elements of \(V(Y)\) is denoted by \(\ast{V(X)}\). A set in \(V(Y)\) that is not internal is called ``external". In particular, we want to look at a nonstandard model \((V(\mathbb N),V(\ast{\mathbb N}),*)\) of arithmetic generated this way. 

We usually notate nonstandard natural numbers with a hat to differentiate them from the standard natural numbers, for example, \(\hat n,\hat m,\hat p\in\ast{\mathbb N}\). Also, for convenience, we identify each algebraic structure with its underlying set, since the operations will always be clear from the context.  Given a nonstandard prime \(\hat p\), we define the pseudo-finite field \(\mathbb F_{\hat p}=\{\hat n\in\ast{\mathbb N}\mid\hat n<\hat p\}\) with addition and multiplication defined modulo \(\hat p\), similarly to the definition of \(\mathbb F_p\) for standard primes \(p\). Note that since \(\mathbb F_{\hat p}\) can be defined using the constant \(\hat p\) in the language of first-order arithmetic, \(\mathbb F_{\hat p}\) is internal.

Each structure that has elements for \(0\) and \(1\) and addition has a unique embedding of \(\mathbb N\) into it. We identify \(\mathbb N\) with all of those embedding, so for example \(\mathbb N\subseteq\ast{\mathbb N}\). Similarly, every field of characteristic \(0\) has a unique copy of \(\mathbb Q\), and these copies are also identified together. It is also sometimes helpful to consider its algebraic closure, for instance \(\bar{\mathbb Q}\subseteq\bar{\mathbb F}_p\). However, many fields that have an embedding of \(\mathbb R\) into them do not have a unique one. When \(\mathbb R\) can be embedded into a field \(\mathbb F_{\hat p}\) its image is notated by \(\mathbb R'\) to mark the fact that it (and the embedding that produced it) might not be unique.


In this paper we present 3 ways to construct external sets by using mostly internal sets. The first ones use countable unions and intersections of internal sets:
\begin{definition}
    A set \(S\subseteq \ast{\mathbb N}\) is called a \emph{\(\sigma\)-set} if it is a union of countably many internal sets.
\end{definition}
\begin{definition}
    Similarly, A set \(S\subseteq\ast{\mathbb N}\) is called a \emph{\(\delta\)-set} if it is an intersection of countably many internal sets.
\end{definition}

The last one combines internal sets with very simple external sets. Namely:
\begin{definition}
    A set \(C\in \mathcal P(\ast{\mathbb N})\) is called a \emph{cut} if it is downwards-closed in the natural order on \(\ast{\mathbb N}\).
\end{definition}
\begin{definition}
    A set \(S\subseteq\ast{\mathbb N}\) is called \emph{almost internal} if it is of the form \(f^{-1}[C]\) for an internal function \(f\) and a cut \(C\).
\end{definition}
\begin{theorem}\label{R-no-sigma-delta}
    No copy of \(\mathbb R\) inside \(\mathbb F_{\hat p}\) can be a \(\sigma\)-set and neither can one be a \(\delta\)-set.
\end{theorem}
\begin{theorem}\label{r-bad}
    No copy of \(\mathbb R\) inside \(\mathbb F_{\hat p}\) can be almost internal
\end{theorem}
\begin{theorem}\label{r-delta-good}
    If \(\bar{\mathbb Q}\cap\mathbb R\) can be embedded in \(\mathbb F_{\hat p}\),
    then if \(\mathbb F_{\hat p}\) contains the square root of \(-1\) there there is an almost internal \(\delta\)-sub-filed of \(\mathbb F_{\hat p}\) that is algebraically closed and of cardinality at least \(\mathfrak c\), and
    if \(\mathbb F_{\hat p}\) does not contain the square root of \(-1\)
    there is an almost internal \(\delta\)-sub­field of \(\mathbb F_{\hat p}\) that is real closed and \(\aleph_1\)-saturated. In particular, in both cases there are at least \(2^\mathfrak c\) ways to embed \(\mathbb R\) into such an almost internal field.
\end{theorem}
\begin{theorem}\label{r-sigma}
    \(\bar{\mathbb Q}\cap\mathbb R\) can be embedded in \(\mathbb F_{\hat p}\),
    then if \(\mathbb F_{\hat p}\) contains the square root of \(-1\) there there is an almost internal \(\sigma\)-sub­filed of \(\mathbb F_{\hat p}\) that is algebraically closed and of cardinality at least \(\mathfrak c\), and
    if \(\mathbb F_{\hat p}\) does not contain the square root of \(-1\)
    there is an almost internal \(\sigma\)-sub­field of \(\mathbb F_{\hat p}\) that is real closed, but no such field can be \(\aleph_1\)-saturated, even if the requirement to be almost internal is omitted. In both cases there are at least \(2^\mathfrak c\) options to embed \(\mathbb R\) in such an almost internal field. 
\end{theorem}
Note that if \(|\ast{\mathbb N}|=\mathfrak c\), which in particular happens if \(|I|=\aleph_0\), the fields in the last two theorems have to also have cardinality \(\mathfrak c\), which means that if \(\mathbb F_{\hat p}\) contains a square root of \(-1\) these fields are isomorphic to the complex numbers and in both cases the bound of \(2^\mathfrak c\) is tight.

In Section \ref{countables} we show that \(\mathbb Q\) and either \(\bar{\mathbb Q}\cap\mathbb R'\) or \(\bar{\mathbb Q}\) are almost internal.
In Section \ref{sigma-delta} we prove some results about \(\sigma\)-sets and \(\delta\)-sets.
In section \ref{R} we prove Theorems \ref{R-no-sigma-delta} and \ref{r-bad} in the case that the field does not contain a square root of \(-1\). This section stands on its own and does not require prior sections in order to understand it.
In section \ref{superfields} we prove Theorems \ref{r-delta-good}
and \ref{r-sigma}
in the case that the field does not contain a square root of \(-1\).
In section \ref{more} we look at other substructures of the field we constructed for the proof of Theorem \ref{r-delta-good} constructed similarly.
In section \ref{complex} we prove theorems \ref{R-no-sigma-delta}, \ref{r-bad}, \ref{r-delta-good} and \ref{r-sigma}
in the case that the field contains a square root of \(-1\).
In addition, we prove some results about cuts in Appendix \ref{cuts}.

\section{Some almost internal sub­fields of \(\mathbb{R}'\)}\label{countables}
First, let us show that some sub­fields of \(\mathbb{R}'\) are almost internal, specifically the rational numbers \(\mathbb{Q}\) and in the case that \(\mathbb F_{\hat p}\) does not contain the square root of \(-1\) also the real algebraic numbers \(\bar{\mathbb Q}\cap\mathbb{R}'\). In the case that \(\mathbb F_{\hat p}\) contains the square root of \(-1\) we show that \(\bar{\mathbb Q}\) is an almost internal subset of it. The cut will be the standard natural numbers \(\mathbb N\) and the functions will be constructed in such a way that we will be able to prove that the result is a field in the first case, and a field that is algebraically closed relative to \(\mathbb{F}_{\hat p}\) in the second case, without relying on the fact that they are in fact \(\mathbb{Q}\) and \(\bar{\mathbb{Q}}\cap\mathbb{R}'\) or \(\bar{\mathbb Q}\) respectively. This will result from the fact that for both functions and for any \(a,b\in\mathbb{F}_{\hat p}\), \(f(a+b),f(ab),f(-a)\), and \(f(a^{-1})\) if \(a\ne 0\), will be bounded using expressions containing only \(f(a)\), \(f(b)\) and standard natural numbers, and in the case of \(\bar{\mathbb Q}\cap\mathbb R\) or \(\bar{\mathbb Q}\) the same will also be true for \(f(x)\) where \(x\) is a solution of some polynomial \(P\) with coefficients \(a_0,\dots,a_n\) and the expression will contain only \(f(a),\dots,f(a_n),n\) and standard natural numbers.

For \(\mathbb{Q}\) we use the following function:
\begin{definition}
    Let \(f_{\mathbb Q}\) be the function that assigns each \(x\in \mathbb{F}_{\hat p}\) the minimal value of \(\max(n, m)\) for \(n,m\in\ast{\mathbb N}\) where either \(x\equiv_{\hat p}\frac nm\) or \(x\equiv_{\hat p}-\frac nm\).
\end{definition}
This function can be defined over all the fields \(\mathbb F_p\) for standard prime \(p\) and therefore it is well defined and internal. For example, in \(\mathbb{F}_{13}\), \(f_{\mathbb Q}(7)=2\) because \(7\equiv_{13}\frac12\).

\begin{proposition}
    For \(f=f_\mathbb Q\), \(f^{-1}[\mathbb N]\cong\mathbb Q\).
\end{proposition}
\begin{proof}
    In one direction, for each \(x\in \mathbb Q\), either \(x=\frac ab\) or \(x=-\frac ab\) for some numbers \(a,b\in\mathbb N\). Therefore \(f(x)\le\max(a, b)\in\mathbb N\). In the other direction, if \(x\in\mathbb F_{\hat p}\) has \(f(x)=n\in\mathbb N\) then \(x\) is either \(\pm\frac nm\) or \(\pm\frac mn\) for \(m\le n\). Therefore \(m\in\mathbb N\) too and so \(x\) is the fraction of two standard integers and so is a rational number.
\end{proof}

\begin{proposition}\label{f-q-properties}
    for \(f=f_{\mathbb Q}\) and for any \(a,b\in\mathbb{F}_{\hat p}\), \(f(a+b)\le 2f(a)f(b)\), \(f(ab)\le f(a)f(b)\), \(f(-a)=f(a)\) and if \(a\ne 0\) then also \(f(a^{-1})=f(a)\).
\end{proposition}
\begin{proof}
    Notate \(a=\frac nm\) and \(b=\frac lk\) for \(n,l\in\ast{\mathbb Z}\) and \(m,k\in\ast{\mathbb N}\) where \(\max(|n|,m)\) and \(\max(|l|,k)\) are minimal. Then \(a+b=\frac{nk+lm}{mk}\) and therefore \begin{align*}
        f(a+b)&\le\max(|nk+lm|,mk)\le\max(|n|k+|l|m,mk)\le\\
        &\le\max(f(a)f(b)+f(b)f(a),f(a)f(b))=\\&=\max(2f(a)f(b),f(a)f(b))=2f(a)f(b).
    \end{align*}
    Similarly, \(ab=\frac{nl}{mk}\) and therefore \(f(ab)\le f(a)f(b)\). Also, \(-a=\frac{-n}{m}\) and so \(f(-a)\le\max(|n|,m)=f(a)\), and with the same logic \(f(a)=f(--a)\le f(-a)\). Similarly, if \(a\ne 0\) then \(m\ne 0\) and so \(a^{-1}=\frac{m}{n}=\pm\frac{m}{|n|}\). Therefore \(f(a^{-1})\le f(a)\) which means that \(f(a^{-1})=f(a)\).
\end{proof}
\begin{proposition}\label{field}
    For any field \(\mathbb F\), ordered semi­ring \(R\), downwards-closed sub-semi­ring \(S\subseteq R\) and function \(f:\mathbb{F}\to R\) satisfying that for each \(a, b\in\mathbb{F}\), \(f(a+b),f(ab),f(-a)\), and \(f(a^{-1})\) if \(a\ne 0\), are bounded by arithmetic terms in only \(f(a)\), \(f(b)\) and constant elements of \(S\), \(f^{-1}[S]\) is also a field.
\end{proposition}
\begin{proof}
    It suffices to prove that for any \(a,b\in f^{-1}[S]\), \(a+b,ab,-a\in f^{-1}[S]\) and if \(a\ne 0\) then also \(a^{-1}\in f^{-1}[S]\). We will show the proof only for the case of a sum, with the other cases proven similarly. Let \(a,b\in f^{-1}[S]\). Then \(f(a+b)\le T(f(a),f(b))\) where \(T\) is an arithmetic term which may only have constants from \(S\). Now, because \(f(a),f(b)\in S\), also \(T(f(a),f(b))\in S\), and so, because \(S\) is downwards closed, \(f(a+b)\in S\) which means that \(a+b\in f^{-1}[S]\).
\end{proof}
\begin{corollary}
    \(\mathbb Q\) is a field.
\end{corollary}
\begin{proof}
    Take \(\mathbb F=\mathbb F_{\hat p}\), \(R=\ast{\mathbb N}\), \(S=\mathbb N\) and \(f=f_\mathbb Q\).
\end{proof}

Now let us describe the function we use for \(\bar {\mathbb Q}\cap\mathbb R'\) and \(\bar{\mathbb Q}\). 
\begin{definition}
    Let \(f_{\bar{\mathbb Q}}\) be the function that maps each \(x\in \mathbb F_{\hat p}\) to the minimal value of \(n+\lceil\log_2\max(m,k)\rceil\) where \(k\ne 0\) and \(kx\) is an eigenvalue of an \(n\times n\) matrix over \(\mathbb F_{\hat p}\) with entries of absolute value at most \(m\).
\end{definition}
This function can again be defined over all the fields \(\mathbb F_p\) for standard prime numbers \(p\) and therefore it is well defined and internal. For example, in \(\mathbb F_{199}\), \(f_{\bar{\mathbb Q}}(10)=3\) because \(20=2\cdot 10\) is an eigenvalue of the matrix \(\begin{pmatrix}
    0&1\\2&0
\end{pmatrix}\) in \(\mathbb F_{199}\), and \(2+\lceil\log_2\max(2,2)\rceil=2+\lceil\log_22\rceil=2+1=3\).

\begin{proposition}\label{f-qbar}
    For \(f=f_{\bar{\mathbb Q}}\) , \(f^{-1}[\mathbb N]=F\) where \(F=\bar{\mathbb Q}\cap\mathbb R'\) if \(\mathbb F_{\hat p}\) does not contain a square root of \(-1\) and \(F=\bar{\mathbb Q}\) if it does.
\end{proposition}
\begin{proof}
    In one direction, if \(x\in F\) then it is an element of \(\mathbb F_{\hat p}\) and there is a standard rational monic polynomial \(P\) of degree \(n\) for which \(P(x)=0\). Notate by \(M\) the companion matrix of \(P\), by \(k\) the lowest common denominator of the entries of \(M\) and by \(m\) the maximal absolute value of an entry of \(kM\). Then \(x\) is an eigenvalue of \(M\) and therefore \(kx\) is an eigenvalue of \(kM\), which means that \(f(x)\le n+\lceil\log_2\max(k,m)\rceil\in \mathbb N\).

    In the other direction, if for some \(x\in\mathbb F_{\hat p}\), \(f(x)=N\in \mathbb N\), then there are \(n\le N\) and \(m,k\le 2^N\) such that \(k\ne 0\) and \(kx\) is an eigenvalue of some \(n\times n\) matrix with entries of absolute value at most \(m\), and because \(n, m, k\) are all standard natural numbers, this means that \(kx\) is an algebraic integer and so that \(x\in\bar{\mathbb Q}\) is algebraic. Additionally, if \(i=\sqrt{-1}\notin\mathbb F_{\hat p}\), because \(\mathbb R'\cong\mathbb R\) its algebraic closure is \(\mathbb R'[i]=\mathbb C'\cong\mathbb C\). Because \(x\in\bar{\mathbb Q}\), it also holds that \(x\in\mathbb C'\) and therefore \(x=a+bi\) for \(a,b\in\mathbb R'\). However, if \(b\ne0\) then we get that \(i=\frac{x-a}{b}\in\mathbb F_{\hat p}\), and therefore \(b=0\) which means that \(x\in\mathbb R'\).
\end{proof}
\begin{proposition}\label{f-qbar-properties}
    For \(f=f_{\bar{\mathbb Q}}\) and any \(a, b\in\mathbb{F}_{\hat p}\), \(f(a+b)\le f(a)f(b)+1\), \(f(ab)\le f(a)f(b)\), \(f(-a)=f(a)\), and if \(a\ne 0\), \(f(a^{-1})\le 2f(a)^2+f(a)\). Furthermore, if \(P\) is an internal nonstandard monic polynomial over \(\mathbb F_{\hat p}\) with coefficients \(a_0,\dots,a_{\hat n-1}\) (i.e. \(P(x)=x^{\hat n}+\sum_{i=0}^{\hat n-1}a_i x^i\)) then for each root \(x\in\mathbb F_{\hat p}\) of \(P\), \(f(x)\le \hat n\prod_{i=0}^{\hat n-1}f(a_i)\).
\end{proposition}
\begin{proof}
    Notate by \(M\) an \(n\times n\) matrix with entries of absolute value at most \(m\) such that \(ka\) is an eigenvalue of \(M\) with eigenvector \(v\) and \(k\ne 0\), and the same for \(M',n',m',k',u\) and \(b\). Then \(M\otimes M'\) is an \(nn'\times nn'\) matrix with entries of absolute value at most \(mm'\) which has \(kk'ab\) as an eigenvalue: \[(M\otimes M')(v\otimes u)=Mv\otimes M'u=kav\otimes k'bu=kk'ab(v\otimes u).\] Therefore \begin{align*}
        f(ab)&\le nn'+\lceil\log_2\max(mm',kk')\rceil\le\\
        &\le nn'+\lceil\log_2\max(m,k)+\log_2\max(m',k')\rceil\le\\
        &\le nn'+\lceil\log_2\max(m,k)\rceil+\lceil\log_2\max(m',k')\rceil\le\\
        &\le nn'+\lceil\log_2\max(m,k)\rceil n'+n\lceil\log_2\max(m',k')\rceil\le\\
        &\le nn'+\lceil\log_2\max(m,k)\rceil n'+\\&+n\lceil\log_2\max(m',k')\rceil+\lceil\log_2\max(m,k)\rceil\lceil\log_2\max(m',k')\rceil=\\
        &=(n+\lceil\log_2\max(m,k)\rceil)(n'+\lceil\log_2\max(m',k')\rceil)=f(a)f(b).
    \end{align*}
    Similarly, \(M\otimes k'I_{n'}+kI_n\otimes M'\) is an \(n\times n'\) matrix with entries of absolute value at most \(mk'+m'k\) which has \(kk'(a+b)\) as an eigenvalue: \begin{align*}
        (M\otimes k'I_{n'}+kI_n\otimes M')(v\otimes u)&=Mv\otimes k'u+kv\otimes M'u=\\&=kav\otimes k'u+kv\otimes k'bu=kk'(a+b)(v\otimes u).
    \end{align*}
    Therefore,
    \begin{align*}
        f(a+b)&\le nn'+\lceil\log_2\max(mk'+m'k,kk')\rceil\\&\le nn'+\lceil\log_2(2\max(m,k)\max(m',k'))\rceil=\\
        &= nn'+\lceil\log_2\max(m,k)+\log_2 \max(m',k')\rceil+1\le f(a)f(b)+1
    \end{align*}
    For negation we can just take \(-M\), for which \(-ka\) is an eigenvector, and get, similarly to the case of \(f_{\mathbb{Q}}\), that \(f(-a)=f(a)\). 
    For the inverse, consider the polynomial 
    \[P(x)=x^np_M\left(\frac 1x\right)=\sum_{0\le i\le d}c_i x^i\] which has degree \(d\le n\). Note that this is an integer polynomial which has \(\frac{1}{ka}\) as a root. Also, since \(M\)'s eigenvalues \(\lambda_i\) are bounded by \(nm\), the coefficient of \(x^i\) in \(p_M(x)\) is bounded by \((mn)^{n-i}\binom ni\le (mn)^{n-i}n^i\le(mn)^n,\) which is therefore a bound on the coefficients \(c_i\) of \(P\) as well. Now, consider the matrix and the vector
    \[M'=k\begin{pmatrix}
        0&c_d&0&\cdots&0\\0&0&c_d&\cdots&0\\\vdots&\vdots&\vdots&\ddots&\vdots\\0&0&0&\cdots&c_d\\-c_0&-c_1&-c_2&\cdots&-c_{d-1}
    \end{pmatrix},v=\begin{pmatrix}
        1\\\frac 1{ka}\\\frac 1{(ka)^2}\\\vdots\\\frac 1{(ka)^{d-1}}
    \end{pmatrix}\]
    and note that they satisfy
    \[M'v=k\begin{pmatrix}
        \frac{c_d}{ka}\\\frac{c_d}{(ka)^2}\\\vdots\\-\sum_{i=0}^{d-1}\frac{c_i}{(ka)^i}
    \end{pmatrix}=\begin{pmatrix}
        k\frac{c_d}{ka}\\k\frac{c_d}{(ka)^2}\\\vdots\\k\frac{c_d}{(ka)^d}
    \end{pmatrix}=\begin{pmatrix}
        \frac{c_d}a\\\frac{c_d}a\frac1{ka}\\\vdots\\\frac{c_d}a\frac{1}{(ka)^{d-1}}
    \end{pmatrix}=\frac{c_d}av,\] and therefore \(c_da^{-1}\) is an eigenvalue of \(M'\). This means that
    \begin{align*}
        f(a^{-1})&\le d+\left\lceil\log_2\left(\max\left(k\max_i(c_i),c_d\right)\right)\right\rceil\le n+\lceil\log_2(k(mn)^n)\rceil\le\\&
        \le \lceil\log_2 k\rceil+n(\lceil\log_2 m\rceil+\lceil\log_2 n\rceil+1)\le\lceil\log_2 k\rceil+n(\lceil\log_2 m\rceil+n)\le\\&
        \le 2f(a)^2+f(a).
    \end{align*}
    
    Finally, for polynomials, for each \(0\le i<\hat n\) notate by \(M_i\) an \(n_i\times n_i\) matrix with entries of absolute value at most \(m_i\) such that \(k_i a_i\) is an eigenvalue of \(M_i\) with eigenvector \(v_i\) and \(k_i\ne 0\). Define \[k=\prod_{i=0}^{\hat n-1}k_i\] and the matrices \(K=kI_{\prod_{i=0}^{\hat n-1}n_i}\) and for each \(0\le i<\hat n\) \[M_i'=\prod_{j=0}^{i-1}k_jI_{\prod_{j=0}^{i-1}n_j}\otimes M_i\otimes\prod_{j=i+1}^{\hat n-1}k_jI_{\prod_{j=i+1}^{\hat n-1}n_j},\] and consider the block matrix
    \[M=\begin{pmatrix}
        0&K&0&\cdots&0\\
        0&0&K&\cdots&0\\
        \vdots&\vdots&\vdots&\ddots&\vdots\\
        0&0&0&\cdots&K\\
        -M_0'&-M_1'&-M_2'&\cdots&-M_{\hat n-1}'
    \end{pmatrix}.\]
    Note that if \(x\) is a solution of \(P\) then \[v=\begin{pmatrix}
        1\\x\\x^2\\\vdots\\x^{\hat n-1}
    \end{pmatrix}\otimes\bigotimes_{i=0}^{\hat n-1}v_i\] satisfies
    \begin{align*}
        Mv&=\begin{pmatrix}
            kx\\kx^2\\kx^3\\\vdots\\-\sum_{i=0}^{\hat n-1}\prod_{j=0}^{i-1}k_j\cdot k_ia_i\cdot\prod_{j=i+1}^{\hat n-1}k_j\cdot x^i
        \end{pmatrix}\otimes\bigotimes_{i=0}^{\hat n-1}v_i=\\&=\begin{pmatrix}
            kx\\kx^2\\kx^3\\\vdots\\-k\sum_{i=0}^{\hat n-1}a_i x^i
        \end{pmatrix}\otimes\bigotimes_{i=0}^{\hat n-1}v_i=\begin{pmatrix}
            kx\\kx^2\\kx^3\\\vdots\\kx^n
        \end{pmatrix}\otimes\bigotimes_{i=0}^{\hat n-1}v_i=kxv,
    \end{align*} and therefore \(kx\) is an eigenvalue of \(M\). This means that
    \begin{align*}
        f(x)&\le\hat n\prod_{i=0}^{\hat n-1}n_i+\left\lceil\log_2\prod_{i=0}^{\hat n-1}\max(m_i,k_i)\right\rceil=
        \\&=\hat n\prod_{i=0}^{\hat n-1}n_i+\left\lceil\sum_{i=0}^{\hat n-1}\log_2\max(m_i,k_i)\right\rceil\le
        \\&\le\hat n\prod_{i=0}^{\hat n-1}n_i+\sum_{i=0}^{\hat n-1}\lceil\log_2\max(m_i,k_i)\rceil\le
        \\&\le\hat n\left(\prod_{i=0}^{\hat n-1}n_i+\sum_{i=0}^{\hat n-1}\lceil\log_2\max(m_i,k_i)\rceil\right)\le
        \\&\le\hat n\prod_{i=0}^{\hat n-1}(n_i+\lceil\log_2\max(m_i,k_i)\rceil)=\hat n\prod_{i=0}^{\hat n-1}f(a_i)
    \end{align*}
\end{proof}
\begin{proposition}\label{subfield-alg}
    For any field \(\mathbb F\), ordered semi­ring \(R\), downwards-closed sub-semi­ring \(S\subseteq R\) and function \(f:\mathbb{F}\to R\) such that for every choice of \(a, b\in\mathbb{F}\), \(f(a+b),f(ab),f(-a)\), and \(f(a^{-1})\) if \(a\ne 0\), are bounded by arithmetic terms involving only \(f(a)\), \(f(b)\) and constants from of \(S\), and also for each \(n\) there is an arithmetic term with \(n+1\) variables using only constants from \(S\) such that for every monic polynomial \(P\) with coefficients \(a_0,\dots a_{n-1}\) and a root \(x\in\mathbb F\) of it, \(f(x)\) is bounded by the value of this term with \(f(a_0),\dots,f(a_{n-1}),n\) substituted into it, then \(f^{-1}[S]\) is also a field which is algebraically closed relative to \(\mathbb F\).
\end{proposition}
\begin{proof}
    The fact that this is a field follows from Proposition \ref{field}. Let us consider a polynomial \(P\) with coefficients in \(f^{-1}[S]\) and a root of it \(x\in\mathbb F\). We can divide by its leading coefficient and get a monic polynomial with coefficients in \(f^{-1}[S]\) which still has \(x\) as a root, and so we can assume w.l.o.g that \(P\) is monic. Therefore \(f(x)\) is bounded by some arithmetic term involving \(f(a_0),\dots,f(a_{n-1})\) and constants from \(S\), which because \(S\) is a semi­ring, means that the value of the expression is in \(S\), and because \(S\) is downwards closed, means that \(f(x)\in S\), and so \(x\in f^{-1}[S]\).
\end{proof}
\begin{corollary}
    Let \(F\) be \(\bar{\mathbb Q}\cap\mathbb R'\) if \(\mathbb F_{\hat p}\) has no square root of \(-1\) and \(\bar{\mathbb Q}\) if it has. Then \(F\) is a field which is algebraically closed relative to \(\mathbb F_{\hat p}\).
\end{corollary}
\begin{proof}
    Choose \(\mathbb F=\mathbb F_{\hat p}\), \(R=\ast{\mathbb N}\), \(S=\mathbb N\) and \(f=f_{\bar{\mathbb Q}}\).
\end{proof}

\section{\(\sigma\)-sets and \(\delta\)-sets}\label{sigma-delta}
\begin{proposition}\label{internal-sigma-delta}
    Any internal set is both a \(\sigma\)-set and a \(\delta\)-set.
\end{proposition}
\begin{proof}
    Let \(S\) be an internal set. \(S=\bigcup_{i\in \mathbb N}S=\bigcap_{i\in \mathbb N}S\) is both a \(\sigma\)-set and a \(\delta\)-set.
\end{proof}
\begin{proposition}
    Every countable subset of \(\ast{\mathbb N}\) is a \(\sigma\)-set.
\end{proposition}
\begin{proof}
    Let \(S\in\mathcal P(\ast{\mathbb N})\) be countable. For any element \(x\in S\), \(\{x\}\) is internal and so \(S=\bigcup_{x\in S}\{x\}\) is a \(\sigma\)-set.
\end{proof}
\begin{proposition}\label{duality-}
    If \(S,T\) are \(\sigma\)-sets then both \(S\cup T\) and \(S\cap T\) are also \(\sigma\)-sets. Additionally, if \(S\) is a \(\sigma\)-set then \(\ast{\mathbb N}\setminus S\) is a \(\delta\)-set and vice-versa.
\end{proposition}
\begin{proof}
    Let \(S=\bigcup_{i\in\mathbb N}S_i\) and \(T=\bigcup_{i\in\mathbb N}T_i\) be \(\sigma\)-sets. Then \[S\cup T=\bigcup_{i\in\mathbb N}(S_i\cup T_i),~ S\cap T=\bigcup_{i,j\in\mathbb N}(S_i\cap T_j), ~\ast{\mathbb N}\setminus S=\bigcap_{i\in\mathbb N}(\ast{\mathbb N}\setminus S_i).\] Similarly, let \(S=\bigcap_{i\in\mathbb N}S_i\) be a \(\delta\)-set, then \[\ast{\mathbb N}\setminus S=\bigcup_{i\in\mathbb N}(\ast{\mathbb N}\setminus S_i).\]
\end{proof}
\begin{corollary}\label{duality}
    If \(S,T\) are \(\delta\)-sets then \(S\cup T\) and \(S\cap T\) are also \(\delta\)-sets. Additionally, if \(S\) is a \(\sigma\)-set and \(T\) is a \(\delta\)-set then \(S\setminus T\) is \(\sigma\) and \(T\setminus S\) is \(\delta\).
\end{corollary}
\begin{proof}
    Let \(S\) and \(T\) be \(\delta\)-sets, then \[S\cup T=\ast{\mathbb N}\setminus((\ast{\mathbb N}\setminus S)\cap (\ast{\mathbb N}\setminus T)),~S\cup T=\ast{\mathbb N}\setminus((\ast{\mathbb N}\setminus S)\cup (\ast{\mathbb N}\setminus T))\] are also both \(\delta\). Similarly, let \(S\) be a \(\sigma\)-set and \(T\) be a \(\delta\)-set, then \[S\setminus T=S\cap(\ast{\mathbb N}\setminus T),~ T\setminus S=T\cap(\ast{\mathbb N}\setminus S)\] are \(\sigma\) and \(\delta\) respectively.
\end{proof}
\begin{theorem}\label{function-sigma}
    For any internal function \(f:\ast{\mathbb N}\to\ast{\mathbb N}\) and any set \(S\), if \(S\) is \(\sigma\) then \(f^{-1}[S]\) and \(f[S]\) are both \(\sigma\).
\end{theorem}
\begin{proof}
    Write \(S=\bigcup_{i\in\mathbb N}S_i\) for internal sets \(S_i\). Then \[f^{-1}[S]=f^{-1}\left[\bigcup_{i\in\mathbb N}S_i\right]=\bigcup_{i\in\mathbb N}f^{-1}[S_i]\] which are all internal sets and therefore \(f^{-1}[S]\) is \(\sigma\). Similarly, \[f[S]=f\left[\bigcup_{i\in\mathbb N}S_i\right]=\bigcup_{i\in\mathbb N}f[S_i]\] and therefore \(f[S]\) is \(\sigma\).
\end{proof}
\begin{corollary}\label{function-delta}
    For any internal function \(f:\ast{\mathbb N}\to\ast{\mathbb N}\) and any set \(S\), if \(S\) is \(\delta\) then \(f^{-1}[S]\) is also \(\delta\).
\end{corollary}
\begin{proof}
    \[f^{-1}[S]=\ast{\mathbb N}\setminus f^{-1}[\ast{\mathbb N}\setminus S].\]
\end{proof}
\begin{theorem}\label{not-both}
    The only sets which are both \(\sigma\) and \(\delta\) are internal.
\end{theorem}
\begin{proof}
    Assume that \(S=\bigcup_{i\in\mathbb N}S_i=\bigcap_{j\in\mathbb N}T_j\) is both \(\sigma\) and \(\delta\) but not internal. Consider the countable set of formulae \(x\not\in S_i\) for all \(i\) and \(x\in T_j\) for all \(j\). Each finite subset of them only includes a finite set of formulae of the form \(x\not\in S_i\), notate it by \(I\). Because \(S\) is not internal but \(\bigcup_{i\in I}S_i\) is they are not equal, and because the former includes the latter, we can find \[x\in S\setminus \bigcup_{i\in I}S_i.\] Such \(x\in\ast{\mathbb N}\) satisfies all the formulae in the finite set. Therefore the original set is finitely satisfiable. However, by \cite[Corollary 4.4.24]{CK} \(\ast{V(\mathbb N)}\) is an \(\aleph_1\)-saturated nonstandard model and so there is an element \(x\in\ast{\mathbb N}\) that satisfies all the formulae, which means that on one hand \(\forall i\in\mathbb N.~x\not\in S_i\), which means that \(x\not\in S\), but on the other \(\forall j\in\mathbb N.~x\in T_j\), which contradicts that.
\end{proof}
Theorem \ref{function-sigma} means it should not come as a surprise that we were able to represent both \(\mathbb Q\) and \(\bar{\mathbb Q}\cap \mathbb R'\) as pre­images of \(\mathbb N\) since they are all countable and so are all \(\sigma\)-sets. However, it turns out that \(\mathbb R'\) is not.

\section{The complicated situation of \(\mathbb R'\)}\label{R}
We assume for the rest of the section that \(\mathbb F_{\hat p}\) does not contain a square root of \(-1\). We detail the required adaptations for the case that it does contain a square root of \(-1\) in Section \ref{complex}
\begin{lemma}\label{finite-subset}
    \(\mathcal P(\mathbb R')\cap\ast{\mathcal P(\mathbb N)}=\mathcal P_{<\aleph_0}(\mathbb R')\), i.e. the internal subsets of \(\mathbb R'\) are exactly its finite subsets.
\end{lemma}
\begin{proof}
    One direction is easy---if \(S\subset\mathbb R'\) is finite, then \(S=\bigcup_{x\in S}\{x\}\in \ast{\mathcal P(\mathbb N)}\).

    For the other direction, first let us show that any unbounded set \(S\subseteq \mathbb R'\) is not internal. Assume that there is an unbounded internal subset \(S\) of \(\mathbb R'\) and assume w.l.o.g that \(S\) is unbounded from above, which means that for each natural number \(n\) there is an element \(a_n\in S\) which is greater than \(n\). However, from \cite[Corollary 4.4.24]{CK} \(\ast{V(\mathbb N)}\) is an \(\aleph_1\)-saturated no standard model and so we can choose \(x\in \mathbb F_{\hat p}\) that satisfies \(x\in S\) as well as \(\exists y_n\in\mathbb F_{\hat p}.~(x-n)=y_n^2\) for all \(n\), since any finite subset of these formulae is satisfied by some \(a_i\). This means that this \(x\in S\subseteq \mathbb R'\), and therefore \(x<n\) for some \(n\), which means that \(0\ne\sqrt{n-x}\in\mathbb R'\subseteq\mathbb F_{\hat p}\), and therefore \(\left(\frac{y_n}{\sqrt{n-x}}\right)^2=\frac{x-n}{n-x}=-1\), which contradicts \(\mathbb F_{\hat p}\) not having a square root of \(-1\).

    Now, let \(S\subset\mathbb R'\) be a bounded infinite set. Because it is bounded, \(S\subseteq[a,b]\) which is a compact set. Therefore, because \(S\) is infinite, it has an accumulation point \(p\in\mathbb R'\), which means that \(\left\{\frac{1}{x-p}\middle|x\in S\setminus\{p\}\right\}\) is unbounded and so it is external, and therefore \(S\) is external as well.
\end{proof}

\begin{proof}[Proof of Theorem \ref{R-no-sigma-delta}]
    First, \(\mathbb R'\) is not \(\sigma\), since if it were, that would mean that \(\mathbb R'=\bigcup_{n\in\mathbb N}S_n\) for \(S_n\in\ast{\mathcal P(\mathbb N})\). This would mean that \(S_n\) are all finite and so \(\mathbb R\) is a countable union of finite sets and so it is countable---a contradiction.

    Second, \(\mathbb R'\) is not \(\delta\) either, since that would mean that \(\mathbb R'=\bigcap_{n\in\mathbb N}S_n\) for \(S_n\in\ast{\mathcal P(\mathbb N)}\). However, from \cite[Corollary 4.4.24]{CK} \(\ast{V(\mathbb N)}\) is an \(\aleph_1\)-saturated nonstandard model, and so there is \(x\in\mathbb F_{\hat p}\) that satisfies both \(x\in S_n\) for all \(n\) and \(\exists y_n\in\mathbb F_{\hat p}.~(x-n)=y_n^2\) for all \(n\). That would mean that there is \(x\in\mathbb R\) that is greater than all the natural numbers---a contradiction.
\end{proof}

\begin{definition}
    For each \(\hat n\in\ast{\mathbb N}\), \([\hat n]=\{\hat m\in\ast{\mathbb N}\mid\hat m<\hat n\}\).
\end{definition}

\begin{proof}[Proof of Theorem \ref{r-bad}]
    Let us show that if \(\mathbb R\) is almost internal we can enumerate it. Let \(f\) be the function that shows this and for each \(x\in\mathbb R\) notate \(\hat n=f(x)\) and observe that both of the sets \(f^{-1}[[\hat n]]\) and \(f^{-1}[\{\hat n\}]\) are internal subsets of \(\mathbb R'\) and so are finite. For every \(\hat n\in\ast{\mathbb N}\) we can enumerate \(f^{-1}[\{\hat n\}]\) by some \(e_{\hat n}:f^{-1}[\{\hat n\}]\to[|f^{-1}[\{\hat n\}]|]\) and define \[n_x=|f^{-1}[[\hat n]]|+e_{\hat n}(x).\]
    Let \(x\ne y\in\mathbb R\). If \(f(x)<f(y)\) then notate \(f(x)=\hat n,f(y)=\hat m\) and note that
    \begin{align*}
        n_x&=\left|f^{-1}\left[[\hat n]\right]\right|+e_{\hat n}(x)<\left|f^{-1}\left[[\hat n]\right]\right|+|f^{-1}[\{\hat n\}]|=\left|f^{-1}\left[[\hat n+1]\right]\right|\le\\&
        \le\left|f^{-1}\left[[\hat m]\right]\right|\le n_y,
    \end{align*}
    and similarly if \(f(x)>f(y)\) then \(n_x>n_y\). Otherwise, notate \(f(x)=f(y)=\hat n\), and note that
    \[n_x=\left|f^{-1}\left[[\hat n]\right]\right|+e_{\hat n}(x)\ne \left|f^{-1}\left[[\hat n]\right]\right|+e_{\hat n}(y)=n_y\] which means that this is indeed an enumeration of \(\mathbb R\)---a contradiction.
\end{proof}
\begin{corollary}
    There are also no cut \(C\) and internal function \(f:\ast{\mathbb N}\to\mathbb F_{\hat p}\) satisfying \(\mathbb R'=f[C]\).
\end{corollary}
\begin{proof}
    Assume there are ones, and note that \(C\ne\ast{\mathbb N}\) since that would imply that \(\mathbb R'=f[\ast{\mathbb N}]\) is internal. Now, define \(g:\ast{\mathbb N}\to\ast{\mathbb N}\) by \(g(x)=\min(f^{-1}[\{x\}])\) if this set is not empty and \(g(x)=\hat N\) for some \(\hat N\not\in C\) if it is, and note that \(g\) is well defined and internal. Now, On one hand, for each \(x\in\mathbb R'\) there is some \(\hat n\in C\) such that \(x=f(\hat n)\), and so \(g(x)\le\hat n\in C\), and on the other hand, if \(g(x)\in C\) then necessarily \(g(x)=\min(f^{-1}[\{x\}])\) which means that \(x=f(g(x))\in\mathbb R'\). Therefore, \(\mathbb R'=g^{-1}[C]\), which contradicts the theorem.
\end{proof}

In fact, Theorem \ref{R-no-sigma-delta} can also be proved from Theorem \ref{r-bad} using the following consequence of \cite[Theorem 4.4.23]{CK}:
\begin{lemma}
    \(\sigma\)-sets ans \(\delta\)-sets are almost internal.
\end{lemma}
\begin{proof}
    Let \(S\) be a \(\sigma\)-set and write \(S=\bigcup_{i\in\mathbb N}S_i\) for internal sets \(S_i\). Define \(f:\mathbb N\to \ast{\mathcal P(\mathbb N)}\) such that for all \(i\), \(f(i)=S_i\). By \cite[Theorem 4.4.23]{CK}, \(f\) can be extended to an internal function \(g:\ast{\mathbb N}\to\ast{\mathcal P(\mathbb N)}\). Define \(h:\mathbb N\to\mathbb N\) such that \(h(\hat n)\) is the minimal \(\hat m\) satisfying \(\hat n\in g(\hat m)\) if one exists, or some nonstandard \(\tilde m\) if one does not. Note that since \(g\) is internal, \(h\) is well-defined and internal as well.

    Let us prove that \(h^{-1}[\mathbb N]=S\). First, for each \(x\in S\), \(x\in S_i\) for some \(i\in\mathbb N\). Therefore \(x\in f(i)=g(i)\) which means that \(h(x)\le i\) and so \(h(x)\in\mathbb N\). Second, for each \(x\in h^{-1}[\mathbb N]\), \(h(x)\ne\tilde m\) and so \(x\in g(h(x))=S_{h(x)}\subseteq S.\)

    Now let \(T\) be a \(\delta\)-set, and let \(S=\ast{\mathbb N}\setminus T\). By Proposition \ref{duality-}, \(S\) is \(\sigma\), and so there is an internal function \(h:\ast{\mathbb N}\to\ast{\mathbb N}\) such that \(h^{-1}[\mathbb N]=S\). Let \(\tilde m\) be any nonstandard integer, and define \[h'(\hat n)=\max(\tilde m-h(\hat n),0),\] \[C=\{\hat m\in\ast{\mathbb N}\mid\forall n\in\mathbb N.~\hat m<\tilde m-n\}.\]

    Then the following are equivalent:
    \[x\in T\;\;\;\;x\notin S\;\;\;\;h(x)\notin\mathbb N\;\;\;\;\tilde m-h'(x)\notin\mathbb N\;\;\;\;h'(x)\in C\;\;\;\;x\in h'^{-1}[C].\]
    Note that since the outputs of \(h'\) are always at most \(\tilde m\) the result of the subtraction in the fourth formula is always a member of \(\ast{\mathbb N}\).
\end{proof}

\section{Characterizations of superfields of \(\mathbb R\)}\label{superfields}
We assume for the following two sections that the field \(\mathbb F_{\hat p}\) is only known to contain a copy of  \(\bar{\mathbb Q}\cap \mathbb R\) and also that it does not contain \(\sqrt{-1}\). These fields can be relatively easily created as an ultra­product of one copy of each finite field using \cite[Theorem 3.10]{Chatzidakis}, which can also be used to show that they can be constructed as a quotient of any ultra­power of the naturals. We prove similar results in the case that the field does contain \(\sqrt{-1}\) in Section \ref{complex}
\begin{definition}
    We define the formula \(\phi_2(\hat n)\) as \[\forall x,y\in f_{\bar{\mathbb Q}}^{-1}[[\hat n]].\exists z\in\mathbb F_{\hat p}.~x^2+y^2=z^2.\] In addition, for any odd number \(d>1\), we define \(\phi_d(\hat n)\) as \[\forall a_{0},\dots a_{d-1}\in f_{\bar{\mathbb Q}}^{-1}[[\hat n]].\exists x\in\mathbb F_{\hat p}.~x^{d}+\sum_{i=0}^{d-1}a_ix^i=0.\]
\end{definition}
Note that these formulae hold for all \(n\in\mathbb N\), as by Proposition \ref{f-qbar}\footnote{The reader is welcome reread the proof and convince themself that it does not rely on the existence of transcendental real numbers in \(\mathbb F_{\hat p}\).}, \(f_{\bar {\mathbb Q}}^{-1}[\mathbb N]\) is isomorphic to \(\bar{\mathbb Q}\cap \mathbb R\) and hence is a real-closed field. However, they are false for \(\hat p\), since \(f_{\bar{\mathbb Q}}^{-1}\left[\left[\hat p\right]\right]=\mathbb F_{\hat p}\), which is an infinite pseudo-finite field and hence is not Pythagorean and has an algebraic extension of every order. Therefore, being first order formulae, for each such \(d\) there is some maximal nonstandard natural \(\hat n_d\) that satisfies \(\phi_d\).
\begin{definition}
    \[C_{\hat p}=\bigcap_{\substack{d=2,3,5,7,9,\dots\\i\in\mathbb N}}\left[\left\lceil\sqrt[i]{\hat n_d}\right\rceil\right]\]
\end{definition}
This \(\delta\)-cut includes \(\mathbb N\) and so from Theorem \ref{not-both} the inclusion is strict.
\begin{definition}
    \[\mathbb R^{\delta}_{\hat p}=f_{\bar{\mathbb Q}}^{-1}\left[C_{\hat p}\right]\] We omit \(\hat p\) when it is clear from the context.
\end{definition}
Note that by Corollary \ref{function-delta} this is a \(\delta\)-set.
\begin{lemma}\label{delta-semiring}
    Every cut of the form \[C=\bigcap_{\substack{\hat n\in S\\i\in\mathbb N}}\left[\left\lceil\sqrt[i]{\hat n}\right\rceil\right]\] for a set \(S\subseteq \ast{\mathbb N}\setminus\mathbb N\) is a semi­ring. In particular,
    \(C_{\hat p}\) is a semi­ring.
\end{lemma}
\begin{proof}
    Because it is a subset of \(\ast{\mathbb N}\), it suffices to show that for each \(a,b\in C\), \(a+b,ab\in C\). Assume w.l.o.g that \(a>b\), and note that \(a+b,ab<\max(a^2,2)\). Therefore, because \(C\) is a cut that includes \(\mathbb N\), it suffices to show that \(a^2\in C\), and that is true since for each \(\hat n\in S\) and \(i\in\mathbb N\), \(a<\left\lceil\sqrt[2i]{\hat n}\right\rceil=\left\lceil\sqrt{\left\lceil\sqrt[i]{\hat n}\right\rceil}\right\rceil\) and therefore \(a^2<\left\lceil\sqrt[i]{\hat n}\right\rceil\).
\end{proof}
\begin{corollary}
    \(\mathbb R^\delta\) is a field which is algebraically closed relative to \(\mathbb F_{\hat p}\).
\end{corollary}
\begin{proof}
    This follows from propositions \ref{f-qbar-properties}\footnote{\label{foo}The reader is welcome reread the proof and convince themself that it works for any field for which the function is well defined.} and \ref{subfield-alg}.
\end{proof}
\begin{definition}\label{field-le}
    For each pair of elements \(a,b\in\mathbb R^\delta\), \(a\le b\) iff there exists some \(x\in\mathbb F_{\hat p}\) such that \(b-a=x^2\).
\end{definition}
\begin{lemma}
    The relation \(\le\) is a (non-strict) total order on \(\mathbb R^\delta\), which makes it an ordered field.
\end{lemma}
\begin{proof} Let us show the 4 axioms of total orders:
    \begin{itemize}
        \item Reflexivity: For each \(a\) we can take \(x=0\) and get that \(a\le a\).
        \item Anti-symmetry: For each \(a,b\), if both \(a\le b\) and \(b\le a\) but \(a\ne b\) this means that there are \(0\ne x,y\in\mathbb F_{\hat p}\) such that \(x^2=a-b=-y^2\). Therefore, \(\left(\frac xy\right)^2=\frac{x^2}{y^2}=\frac{-y^2}{y^2}=-1\), contradicting \(\mathbb F_{\hat p}\) not having a square root of \(-1\).
        \item Transitivity: Assume \(a\le b\le c\) for elements of \(\mathbb R^\delta\). This means there are \(x,y\in\mathbb F_{\hat p}\) such that \(x^2=b-a\) and \(y^2=c-b\). Since \(\mathbb R^\delta\) is a field which is algebraically closed relative to \(\mathbb F_{\hat p}\), \(x,y\in\mathbb R^\delta\). Therefore, \(f_{\bar{\mathbb Q}}(x),f_{\bar{\mathbb Q}}(y)\in C_{\hat p}\) which means they are smaller than \(\hat n_2\). Finally, since \(\phi_2\) holds for \(\hat n_2\), there is \(z\in\mathbb F_{\hat p}\) such that \(z^2=x^2+y^2=b-a+c-b=c-a\), meaning that indeed \(a\le c\).
        \item Totality: Since \(\mathbb F_{\hat p}\) does not contain a square root of \(-1\), \(\hat p\equiv_4 3\). Let \(a,b\in\mathbb R^\delta\) and look at \(x=(a-b)^{\frac{\hat p+1}4}\). \(x^4=(a-b)^{\hat p+1}=(a-b)^2\) and therefore either \(x^2=a-b\) or \(x^2=b-a\), meaning that either \(a\le b\) or \(b\le a\).
    \end{itemize}
    Furthermore, if \(a\le b\) then for each \(c\), \((b+c)-(a+c)=b-a\) meaning that \(a+c\le b+c\). Also, if \(a,b\ge 0\) then there are \(x,y\) such that \(a=x^2\) and \(b=y^2\), meaning that \(ab=(xy)^2\), thereby proving that this relation indeed makes \(\mathbb R^\delta\) an ordered field.
\end{proof}
\begin{theorem}\label{real-closed}
    For each cut \(C\subseteq C_{\hat p}\) that is also a semi­ring, \(\mathbb R^C\coloneq f_{\bar{\mathbb Q}}^{-1}[C]\) is a real-closed field with the relation \(\le\). In particular, \(\mathbb R^\delta\) is a real-closed field.
\end{theorem}
\begin{proof}
    From propositions \ref{f-qbar-properties}\footnote{The reader is welcome reread the proof and convince themself that it works for any field for which the function is well defined.}
    and \ref{subfield-alg} this is a field which is algebraically closed relative to \(\mathbb F_{\hat p}\), and since it is a sub­field of \(\mathbb R^\delta\), the previous lemma implies that it is ordered by the relation \(\le\). It remains to show that every positive element has a square root and that every monic polynomial of odd degree has a solution. First, let \(a\ge 0\) be an element of \(\mathbb R^C\). Since \(a\ge 0\), there is \(x\in\mathbb R^\delta\subseteq \mathbb F_{\hat p}\) such that \(x^2=a\), and since \(\mathbb R^C\) is algebraically closed relative to \(\mathbb F_{\hat p}\), \(x\in \mathbb R^C\). Second, let \[p(x)=x^d+\sum_{i=0}^{d-1}a_ix^i\] be a polynomial of odd degree with coefficients in \(\mathbb R^C\). Since \(a_i\in\mathbb R^C\), it follows that \(f_{\bar {\mathbb Q}}(a_i)\in C\), and in particular \[\hat m\coloneq\max_{i}f_{\bar {\mathbb Q}}(a_i)\in C\subseteq C_{\hat p}.\] Therefore, since \(C_{\hat p}\) is a semi­ring, \(\hat n=\hat m+1\in C_{\hat p}\), which means that \(\phi_{d}(\hat n)\) holds. Therefore there is some \(x\in\mathbb F_{\hat p}\) such that \(p(x)=0\), and since \(\mathbb R^C\) is algebraically closed relative to \(\mathbb F_{\hat p}\), \(x\in\mathbb R^C\).
\end{proof}
The following implies the first part of Theorem \ref{r-delta-good}:
\begin{theorem}\label{saturated-field}
    If \(C\subseteq C_{\hat p}\) is a cut that is a semi­ring and also not \(\sigma\) then \(\mathbb R^C\) is an \(\aleph_1\)-saturated real-closed field. In particular \(\mathbb R^\delta\) is an \(\aleph_1\)-saturated real-closed field.
\end{theorem}
\begin{proof}
    It is a real-closed field from Lemma \ref{real-closed}. Let \(x_0,x_1,\dots\) be some countable set of elements of \(\mathbb R^C\) and let \(\Gamma\) be a type in one variable over \(\mathbb R^C_{x_0,x_1,\dots}\). Since the theory of real-closed fields is countable, this type is also countable, which means that we can enumerate its finite subsets \(\mathcal P_{<\aleph_0}(\Gamma)=\{S_0,S_1,\dots\}\). For each \(S_i\) choose some \(y_i\in\mathbb R^C\) that realizes it. Note that \[C'=\bigcup_{i\in\mathbb N}[f_{\bar {\mathbb Q}}(y_i)]\subsetneq C\] since for each \(y_i\), \(f_{\bar {\mathbb Q}}(y_i)\in C\), a cut that is not \(\sigma\), while \(C'\) is \(\sigma\). Therefore, there is some \(\hat n\in C\) that satisfies \(\hat n\ge f_{\bar {\mathbb Q}}(y_i)\) for all \(i\) and therefore there is a countable set of bounded quantifier formulae that contains all the formulae in \(\Gamma\) that have no quantifiers where every time \(\le\) appears it is replaced by Definition \ref{field-le}, as well as the formulae \(x\in\mathbb F_{\hat p}\) and \(f_{\bar {\mathbb Q}}(x)\le\hat n\). From \cite[Corollary 4.4.24]{CK}, \(\ast{V(\mathbb N)}\) is an \(\aleph_1\)-saturated nonstandard model and so this type is realized by some \(x\in\mathbb F_{\hat p}\). This \(x\) also satisfies \(f_{\bar{\mathbb Q}}(x)\le\hat n\) and therefore \(x\in \mathbb R^C\). To show that \(x\) realizes the type \(\Gamma\) we consider any formula \(\phi\in\Gamma\). By \cite[Theorem 8.4.4]{Hodges} the theory of real-closed fields has quantifier elimination, and so \(\phi\) is equivalent to a formula \(\psi\in\Gamma\), that has no quantifiers. Therefore if all the appearances of \(\le\) in \(\psi\) are replaced by its definition the resulting formula satisfies \(\psi'\in\Gamma'\). This means that it is satisfied by \(x\) and therefore \(\phi\) is also satisfied by \(x\).
\end{proof}
\begin{lemma}\label{alg-generate}
    There is a set \(G\) of algebraically independent real numbers that algebraically generates \(\mathbb C\), i.e. \(\overline{\mathbb Q(G)}=\mathbb C\).
\end{lemma}
\begin{proof}
    This follows from Zorn's lemma. Consider \(S\) to be the set of all algebraically independent sets of (transcendental) real numbers, and order it by the subset relation. \(\emptyset\in S\) and so it is nonempty, and every chain in \(S\) has a union which is an upper bound, since every set that is not algebraically independent has a finite subset that is also not algebraically independent. Therefore there is a maximal element \(G\in S\). Let \(z\in\mathbb C\). Since \(G\) is maximal, both of the real numbers \(z+\bar z\) and \(z\bar z\) are algebraically dependent on \(G\), and therefore \((x-z)(x-\bar z)\) has \(z\) as a root but coefficients that are all algebraically dependent on \(G\), which means that \(z\in \overline{\mathbb Q(G)}\). Therefore \(\overline{\mathbb Q(G)}=\mathbb C\).
\end{proof}
The following implies the second part of theorem \ref{r-delta-good}:
\begin{corollary}\label{R-copy}
    For each nonstandard natural number \(\hat n\in\ast{\mathbb N}\setminus\mathbb N\), there are at last \(2^\mathfrak c\) copies of \(\mathbb R\) in \(f_{\bar {\mathbb Q}}^{-1}[[\hat n]]\).
\end{corollary}
\begin{proof}
    Let us define a cut \[C=\bigcap_{i\in\mathbb N}\left[\left\lceil\sqrt[i]{\hat n}\right\rceil\right].\] This cut is \(\delta\) and therefore from theorem \ref{not-both} it is not \(\sigma\). Also, From lemma \ref{delta-semiring} this is a semi­ring. If \(C\not\subseteq C_{\hat p}\), because they are both cuts this necessarily means that \(C_{\hat p}\subseteq C\) and so we can continue with \(C_{\hat p}\) as \(C\). Therefore, from Theorem \ref{saturated-field}, \(\mathbb R^C=f_{\bar{\mathbb Q}}^{-1}[C]\subseteq f_{\bar{\mathbb Q}}^{-1}[[n]]\) is an \(\aleph_1\)-saturated real-closed field.
    
    Let \(G\) an algebraically independent set of real numbers that algebraically generates \(\mathbb C\). Let us define a function \(f:G\to \mathbb R^C\). For each \(g\in G\), approximate it from below by one set of rationals \(q_i\to g\) and from above by another \(r_i\to g\). Note that there is a type in one variable \(y\) over \(\mathbb R^C\) that contains the formulae \(q_i\le y\le r_i\), since for any finite subset of them the maximal \(q_i\) mentioned in the set satisfies it. Therefore, by \(\aleph_1\)-saturation there is some \(y\in\mathbb R^C\) that satisfies all of them, define \(f(g)=y\). Since \(G\) is algebraically independent and \(\mathbb R^C\) is a field, this function extends to an embedding \(f:\mathbb Q(G)\hookrightarrow\mathbb R^C\).
    
    Now, let us extend this embedding to the entire \(\mathbb R\). Let \(a\in\mathbb R\). Since \(\mathbb R\subset\mathbb C=\overline{\mathbb Q(G)}\), it is the zero of a unique minimal monic polynomial \(m(x)\) with coefficients in \(\mathbb Q(G)\). If this polynomial has a double root then \(\gcd(m(x),m'(x))\) for its derivative \(m'\) is a polynomial of degree \(0<d<\deg m\) dividing \(m(x)\) and therefore either it or \(m(x)\) divided by it would be a polynomial of lower degree which has \(a\) as a root, contradicting its minimality, therefore \(m(x)\) has no double roots. Furthermore, this polynomial's coefficients only depend on a finite number of generators \(g_0,\dots,g_{n-1}\in G\), we notate by \(m[h_0\dots,h_{n-1}](x)\) the rational function in \(n+1\) variables over \(\mathbb Q\) that when substituting \(g_0,\dots,g_{n-1}\) into it we get the polynomial \(m(x)\). Now, since \(m(x)\) has no double roots, changing its coefficients a little would still lead to a polynomial that has the same number of roots, and therefore there are some rational numbers \(q_i<g_i<r_i\) such that for every real \(q_i<h_i<r_i\), \(m[h_0,\dots,h_{n-1}](x)\) has the same number of roots. It follows from Theorems 8.4.4, 8.3.1(e\(\to\)a), and 8.3.5 in \cite{Hodges} and from the fact that \(\bar{\mathbb Q}\cap\mathbb R\) can be embedded in every real closed field that the theory of real-closed fields is complete. Therefore, since this is a first order sentence that is true in \(\mathbb R\), it is true in any real-closed field, and in particular in \(\mathbb R^C\). Define \(f(a)\) to be the \(i\)th largest root of \(m[f(g_0),\dots,f(g_{n-1})](x)\) where \(a\) is the \(i\)th largest root of \(m(x)\).
    
    It is left to show that this function respects addition and multiplication. Let \(\star\) be either \(+\) or \(\cdot\), let \(a,b,c\) be 3 real numbers such that \(a\star b=c\) and notate by \(m_a(x),m_b(x),m_c(x)\) the minimal polynomials of \(a,b,c\) respectively, and by \(i,j,k\) the indices such that \(a\) is the \(i\)th largest root of \(m_a(x)\) etc. These polynomials depend only on a finite number of generators \(g_0,\dots,g_{n-1}\in G\). Define \(m_a[h_0,\dots,h_{n-1}](x)\) like before and similarly for \(b,c\) and note that since \(g_0,\dots,g_{n-1}\) are algebraically independent, for any choice of \(h_0,\dots,h_{n-1}\) there are (maybe complex) roots \(\alpha,\beta,\gamma\) of \(m_a[h_0,\dots,h_{n-1}](x),m_b[h_0,\dots,h_{n-1}](x)\) and \(m_c[h_0,\dots,h_{n-1}](x)\) respectively such that \(\alpha\star\beta=\gamma\). In particular, since \(m_a(x),m_b(x),\) and \(m_c(x)\) do not have double roots, there are some rational numbers \(q_\ell< g_\ell< r_\ell\) such that for all real \(q_\ell< h_\ell< r_\ell\), the \(i\)th largest root \(\alpha\) of \(m_a[h_0,\dots,h_{n-1}](x)\) and similarly for \(\beta\) and \(\gamma\) satisfy \(\alpha\star\beta=\gamma\). Again, since this is a first order sentence that is true in \(\mathbb R\) it is also true in \(\mathbb R^C\), and therefore \(f(a)\star f(b)=f(c)\), meaning this is a field embedding \(f:\mathbb R\hookrightarrow \mathbb R^C\subseteq f_{\bar{\mathbb Q}}^{-1}[[\hat n]]\).

    To see that there are at least \(2^\mathfrak c\) such copies, note that if we choose some \(0<\varepsilon\in\mathbb R^C\) that is smaller than any positive rational, which exists as a result of Theorem \ref{saturated-field}, the proof does not change if for some subset \(S\) of \(G\) we increase \(f(g)\) for each \(g\in S\) by \(\varepsilon\). Therefore, since \(|G|=\mathfrak c\), there are at least \(2^{\mathfrak c}\) copies of \(\mathbb R\) inside \(f_{\bar{\mathbb Q}}^{-1}[[\hat n]]\). 
\end{proof}

Note that no embedding \(f:\mathbb R\hookrightarrow\mathbb F_{\hat p}\) is the ``simplest" in the sense of having a minimal value of \(f_{\bar{\mathbb Q}}\), since for each such \(f\) we can find a ``simpler" copy of \(\mathbb R\) in \(f_{\bar{\mathbb Q}}^{-1}[[f(e)]]\).

\begin{proof}[Proof of Theorem \ref{r-sigma}]
    Let \(\hat n\in C_{\hat p}\setminus\mathbb N\), and let \(C=\bigcup_{i\in\mathbb N}\left[\hat n^i\right]\). This \(\sigma\)-cut is a semi­ring, since it contains \(2\) and for each \(a\in C\), \(a<\hat n^i\) for some \(i\) which means that \(a^2<n^{2i}\) and so \(a^2\in C\). It is also a sub-cut of \(C_{\hat p}\), since by Lemma \ref{delta-semiring}, \(\hat n^i\in C_{\hat p}\) for all \(i\). Therefore by Theorem \ref{real-closed}, \(\mathbb R^C\) is a real-closed field and by Theorem \ref{function-sigma} it is \(\sigma\).

    Let \(F\) be a \(\sigma\) real-closed sub­field of \(\mathbb F_{\hat p}\), and let us show that it is not \(\aleph_1\)-saturated. Since it is \(\sigma\), \(F=\bigcup_{i\in\mathbb N}S_i\) for internal \(S_i\), and since the order from Definition \ref{field-le} is internal and coincides with the order of \(F\) as a real-closed field, each \(S_i\) has a maximal element \(m_i\). Let us consider a type in one variable over \(F\) that contains the formula \(x>m_i\) for each \(i\). Every finite set of these formulae only mentions finitely many \(m_i\)s and their maximum \(+ 1\) satisfies this subset, which means that they can indeed be completed to a type. However they cannot all be satisfied at once, since for each \(x\in F\), \(x\in S_i\) for some \(i\) and therefore \(x\le m_i\).
\end{proof}

\section{Other algebraic structures included in \(\mathbb R^\delta\)}\label{more}
We can require that not only any standard odd degree polynomial is solvable, but also nonstandard odd degree polynomials, by replacing all \(\phi_i\) for odd \(i\) with one \(\phi_f\) for some (weakly) increasing unbounded \(f:\mathbb N\to\mathbb N\) where \(\phi_f(\hat n)\) requires that every nonstandard polynomial with coefficients in \(f_{\mathbb Q}^{-1}[[\hat n]]\) and odd degree up to \(\ast f(\hat n)\) has a root in \(\mathbb F_{\hat p}\): \begin{align*}
    &\forall \hat d<\ast f(\hat n), s\in\ast{(\mathbb N^\mathbb N)}.~\left(\hat d\equiv_21\wedge\forall \hat m.~s(\hat m)<\hat p\wedge f_{\bar{\mathbb Q}}(s(\hat m))<\hat n\right)\to\\&\to\exists x\in\mathbb F_{\hat p}.~x^{\hat d}+\sum_{\hat m<\hat d}s(\hat m)x^{\hat m}=0.
\end{align*} However, in this case the cuts \(C\) do not only need to be semi­rings, but we should also require that \(\forall x\in C.~x^{\ast f(x)}\in C\), which is stronger for any unbounded function \(f\):
\begin{proposition}
    If a cut \(\mathbb N\subseteq C\) satisfies \(\forall x\in C.~x^{\ast f(x)}\in C\) for some unbounded increasing \(f:\mathbb N\to\mathbb N\), then \(C\) is a semi­ring.
\end{proposition}
\begin{proof}
    Let \(a,b\in C\). If they are both in \(\mathbb N\) then \(a+b,ab\in\mathbb N\subseteq C\). Otherwise, \[a+b,ab<\max(a,b)^2<\max(a,b)^{\ast f(\max(a,b))}\in C.\]
\end{proof}
\begin{proposition}
    For every function \(f\) there is a \(\delta\)-cut \(C_{\hat p}(f)\) which is a semi­ring such that every element of it \(x\in C_{\hat p}(f)\) satisfies both \(\phi_2\) and \(\phi_f\), and also \(x^{\ast f(x)}\in C\).
\end{proposition}
\begin{proof}
    Because both \(\phi_2\) and \(\phi_f\) are equivalent to first order formulae, the set of nonstandard natural numbers that satisfy both of them is some \(S\in \ast{\mathcal P(\mathbb N)}\) and therefore it has a maximal element \(\hat n\). Note that for all \(\hat m<\hat n\), \(\phi_2(\hat m)\) and because \(f\) is increasing the same is true for \(\phi_f\). Therefore \(S=[\hat n+1]\). Now, let \(g(x)=x^{\ast f(x)}\) and define \[C_{\hat p}(f)=\{\hat m\in\ast{\mathbb N}\mid\forall i\in\mathbb N.~g^i(\hat m)\le\hat n\}\] where exponentiation denotes repeated application. Note that this cut is \(\delta\), and because \(\ast f\) sends elements of \(\mathbb N\) to \(\mathbb N\), \(\mathbb N\subseteq C_{\hat p}(f)\). Also, for every \(x\in C\), and every \(i\in\mathbb N\), \[g^i\left(x^{\ast f(x)}\right)=g^i(g(x))=g^{i+1}(x),\] and therefore \(x^{\ast f(x)}\in C_{\hat p}(f)\) as well, meaning it is also a semi­ring. Finally, \(x=g^0(x)\le n\) as well meaning that \(x\) satisfies both \(\phi_2\) and \(\phi_f\).
\end{proof}
\begin{proposition}
    \(C_{\hat p}(f)\subseteq C_{\hat p}\).
\end{proposition}
\begin{proof}
    Let \(\hat m\in C_{\hat p}(f)\),
    let \(d\in\mathbb N\) for which \(\phi_d\) is defined and let us show that \(\hat m\) satisfies \(\phi_d\). If \(d=2\) then this is true by the definition of \(C_{\hat p}(f)\). Otherwise, \(d\ge 3\) is odd. Let \(a_0,\dots,a_{d-1}\in f_{\bar{\mathbb Q}}^{-1}[[\hat m]]\). Choose \(\hat d=d\), \(s(i)=a_i\) for \(i<d\) and \(s(i)=0\) for \(i\ge d\), and note that \(\phi_f\) guarantees that \[x^d+\sum_{m<d}a_m x^m\] is zero for some \(x\in\mathbb F_{\hat p}\), which means that \(\hat m\) satisfies \(\phi_d\). This is true for any element of \(C_{\hat p}(f)\) and therefore, since it is a semi­ring, also true for \(\hat m^i\) for each \(i\in\mathbb N\). Therefore \(\hat m^i< \hat n_d\), i.e. \(\hat m<\left\lceil \sqrt[i]{\hat n_d}\right\rceil\) for all \(i,d\), which means that \(\hat m\in C_{\hat p}\).
\end{proof}
\begin{definition}
    \(\mathbb R^\delta_{\hat p}(f)\coloneq f_{\bar{\mathbb Q}}^{-1}[C_{\hat p}(f)]\). As before, \(\hat p\) is omitted when it is clear from the context.
\end{definition}
\begin{proposition}
    \(\mathbb R^\delta(f)\) is an \(\aleph_1\)-saturated real-closed field that also contains the root of every nonstandard polynomial \(P\in\ast{V(\mathbb N)}\) of odd degree of at most \(\ast f(\hat n)\) for any \(\hat n\in C_{\hat p}(f)\).
\end{proposition}
\begin{proof}
    From Theorem \ref{saturated-field} this is an \(\aleph_1\)-saturated real-closed field. Let \(\hat n\in C_{\hat p}(f)\) and let \[P(x)=\sum_{i\le \hat d}a_i x^i\] for some nonstandard \(\hat d\le\ast f(\hat n)\) be a nonstandard polynomial over \(\mathbb R^\delta(f)\). Define \(s(i)=\frac{a_i}{a_{\hat d}}\) for \(i<\hat d\) and \(s(i)=0\) for \(i\ge\hat d\) and note that since \(P\) is a nonstandard polynomial, \(s\) is internal. Therefore, there is some \[\hat m=\max_{i\le\hat n}f_{\bar{\mathbb Q}}(s(i)).\] This \(\hat m=f_{\bar{\mathbb Q}}\left(\frac{a_i}{a_{\hat n}}\right)\) for some \(i\) and therefore \(\hat m\in C_{\hat p}(f)\). Thus by \(\phi_f[\max(\hat n,\hat m)]\) there is some \(x\in\mathbb F_{\hat p}\) that satisfies \[x^{\hat d}+\sum_{i<\hat d}\frac{a_i}{a_{\hat d}}x^i=0\] and therefore \(P(x)=0\). Finally, from Proposition \ref{f-qbar-properties}, \[f_{\bar{\mathbb Q}}(x)\le\hat d\hat m^{\hat d}<\left(\hat m^{\hat d}\right)^2\in C_{\hat p}(f)\] and therefore \(x\in\mathbb R^\delta(f)\).
\end{proof}
\begin{proposition}
    If \(f'\) grows faster than \(f\) then \(C_{\hat p}(f')\subseteq C_{\hat p}(f)\).
\end{proposition}
\begin{proof}
    Note that if \(f'\) grows faster than \(f\) then for every non-standard natural \(\hat n\in\ast{\mathbb N}\setminus\mathbb N\), \(\ast{f'}(\hat n)>\ast f(n)\). Therefore if \(\phi_{f'}[\hat n]\) is true then so is \(\phi_f(\hat n)\), which means that the maximal element that satisfies \(\phi_{f'}\), call it \(\hat n_{f'}\) is smaller than or equal to the maximal one that satisfies \(\phi_f\), \(\hat n_f\). Also, if we define \(g(x)=x^{\ast f(x)}\) and \(g'(x)=x^{\ast{f'}(x)}\), then \(g'(\hat n)>g(\hat n)\) for all \(\hat n\in\ast{\mathbb N}\setminus\mathbb N\). Therefore, if \(\hat n\in C_{\hat p}(f')\) then for all \(i\in\mathbb N\),
    \[{g}^i(\hat n)<{g'}^i(\hat n)<\hat n_{f'}\le\hat n_{f}\] and so \(\hat n\in C_{\hat p}(f)\), which means that \(C_{\hat p}(f')\subseteq C_{\hat p}(f)\).
\end{proof}

Another possible restriction is to only allow some of the parameters of the matrix to be nonstandard naturals while keeping the rest of them in \(\mathbb N\), or even \(1\). For convenience, we will only consider this restriction here for cuts that, for each element \(\hat n\) of them, also contain \(2^{\hat n}\), although similar conclusions also hold for other cuts. It does not make sense to restrict the entry bound \(m\) more than the denominator \(k\) because then addition would break, and it also does not make sense to restrict it more than the matrix's size \(n\), since the matrices of the form \[\begin{pmatrix}
    1&\dots&1&1&0&\dots&0&0\\
    0&1&\dots&1&1&0&\dots&0\\
    \vdots&\ddots&\ddots&\ddots&\ddots&\ddots&\ddots&\vdots\\
    0&\dots&0&1&\dots&1&1&0\\
    0&\dots&0&0&1&\dots&1&1\\
    1&0&\dots&0&0&1&\dots&1\\
    \vdots&\ddots&\ddots&\ddots&\ddots&\ddots&\ddots&\vdots\\
    1&\dots&1&0&\dots&0&0&1
\end{pmatrix}\] have an eigenvector of all ones with an eigenvalue of the number of ones in each row or column, which can be any number between \(0\) and the number of rows of this matrix, which means that replacing each number in the matrix with a block like this produces a matrix with the same eigenvalues (and maybe more).

Restricting all the parameters to be at most \(1\) just gives us the set \(\{-1, 0, 1\}\), while proposition \ref{f-qbar} shows that restricting all of them to be standard naturals gives us the copy of \(\bar{\mathbb Q}\cap\mathbb R\) present in the field. The reader may check that allowing the entries to be any natural number but keeping just the size, just the denominator or both \(1\) gives us \(\mathbb Q,\) the ring of real algebraic integers \(\bar{\mathbb Z}\cap\mathbb R\) and \(\mathbb Z\) respectively.

Note also that when restricting only the matrix size to \(1\) from propositions \ref{f-q-properties} and \ref{field} we still get a field, which is in fact the field \(\Frac(C)\) of fractions of the ring that contains both the elements of the cut \(C\) chosen and their opposites. Also, when restricting the matrix size to natural numbers we get a field that is algebraically closed relative to \(\mathbb F_{\hat p}\), since the proof of Proposition \ref{f-qbar-properties} still works, and since it is a sub­field of \(\mathbb R^\delta\) this means that it is also real closed. Also, since in the proof of Theorem \ref{R-copy} the embeddings of the generators can be chosen to be in \(\Frac(C)\), there is also a copy of \(\mathbb R\) in that field.

Going the other way, and restricting the denominator, only gives us rings and not fields. Restricting it to \(1\) gives us a ring \(\bar C_{\hat p}\)\footnote{Even though it is notated the same as an algebraic closure, this is larger than the standard algebraic closure of \(C_{\hat p}\), since it also contains roots of nonstandard monic polynomials.} that is similar to the ring of real algebraic integers in the sense that it is ordered, closed under square roots of positive elements, and every monic polynomial of odd degree has a root, again from a careful consideration of the proof of Proposition \ref{f-qbar-properties} and Theorem \ref{real-closed}. In fact the connection between these rings goes even deeper:
\begin{proposition}
    \(\mathbb R^C\) is the field of fractions of \(\bar C_{\hat p}\).
\end{proposition}
\begin{proof}
    On one hand \(\bar C_{\hat p}\subseteq\mathbb R^C\), and therefore \(\Frac\left(\bar C_{\hat p}\right)\subseteq\mathbb R^C\) as well. On the other hand, given \(x\in\mathbb R^C\) there is some \(\hat n\times\hat n\) matrix \(M\in\ast{V(\mathbb N)}\) with entries up to \(\hat m\) such that \(\hat kx\) is an eigenvalue of it for some \(\hat k\ne 0\) and that \(\hat n+\lceil\log_2(\max(m,k))\rceil\in C\). Therefore \(\hat kx\in \bar C_{\hat p}\) and \(\hat k\) is as well since it is an eigenvalue of the matrix \(\begin{pmatrix}\hat k\end{pmatrix}\), which means that \(x=\frac{\hat kx}{\hat k}\in\Frac\left(\bar C_{\hat p}\right)\). Therefore \(\mathbb R^C\subseteq \Frac\left(\bar C_{\hat p}\right)\) which means that they are the same.
\end{proof}
However, this ring has no copy of \(\mathbb R\) in it since it has no copy of \(\mathbb Q\) in it. This is a consequence of the following proposition:
\begin{proposition}
    There is no element of \(\emph{Frac}(C)\) represented in \(\bar C_{\hat p}\) which is not in \(C\) itself.
\end{proposition}
\begin{proof}
    Assume w.l.o.g that the element in question is positive. Let \(\hat a, \hat b\in C\) be two coprime elements for which there is some \(x\in\bar C_{\hat p}\) such that \(x\hat b\equiv_{\hat p}\hat a\). Because \(x\in \bar C_{\hat p}\) there is an \(\hat n\times\hat n\) matrix \(M\in\ast{V(\mathbb N)}\) with coefficients of absolute value up to \(\hat m\) such that \(x\) is an eigenvalue of \(M\) i.e. \(p_M(x)=0\). Therefore \(b^{\hat n}p_M\left(\frac ab\right)\equiv_{\hat p}0\). However, this number can be bounded. First of all, as shown during the proof of Proposition \ref{f-qbar-properties}, the coefficients of \(p_M\) are bounded by \((\hat m\hat n)^{\hat n}\). Therefore, this entire number is bounded in absolute value by \[\hat n\max(a,b)^{\hat n}(\hat m\hat n)^{\hat n}<2^{\lceil\log_2 \hat n\rceil+\lceil\log_2 \max(a,b)\rceil\hat n+\lceil\log_2\hat m\rceil\hat n}\in C\] which in particular means that it is less than \(\hat p\) and therefore \(0\). However, this number is of the form \(a^{\hat n}+by\), which means that \(b|a^{\hat n}\), and since they are coprime it follows that \(b=1\).
\end{proof}

Still restricting the denominator but only to natural numbers gives us a larger ring, \(\mathbb N^{-1}\bar C_{\hat p}\), which contains all the standard real algebraic integers since they are \(f_{\bar{\mathbb Q}}^{-1}[\mathbb N]\).
\begin{lemma}\label{alg-integer}
    For any finite set of monic polynomials \(\{p_i\mid i<n\}\) over \(\bar{\mathbb Q}\) and every segment \(I\subset\mathbb R\) of positive length there is an algebraic integer \(\alpha\in I\) such that \(\frac 1{p_i(\alpha)}\) for each \(i<n\) is also an algebraic integer.
\end{lemma}
The proof of the lemma is deferred to Appendix \ref{alg-integer-lemma-proof}.
\begin{proposition}
    Assuming CH, the real numbers can be embedded in \(\mathbb N^{-1}\bar C_{\hat p}\).
\end{proposition}
\begin{proof}
    Let us call a rational function 
    on the variables \(x_i\), where \(i\in I\) for some totally ordered set \(I\), and over a field \(\mathbb F\) ``last-monic" if it is a monic polynomial in \(x_i\) over \(\mathbb F(\{x_j\mid j<i\})\) for some index \(i\). 

    Let \(G\) be a set of algebraically-independent real numbers that algebraically generates \(\mathbb C\), \(\overline{\mathbb Q(G)}=\mathbb C\), which exists from Lemma \ref{alg-generate}.
    From CH, we can write \(G=\{g_i\mid i\in\omega_1\}\). 
    Let \(f_{\bar{\mathbb Z}}\) be the same as \(f_{\bar{\mathbb Q}}\) but with the denominator \(k\) restricted to only be the constant \(1\) and note that by definition \(f_{\bar{\mathbb Z}}^{-1}[C]=\bar C_{\hat p}\). In addition, let \(\hat m\in C\setminus\mathbb N\) and let us consider the following bounded quantifier formulae:
    \begin{enumerate}
        \item\label{in-set} For each \(i\in\omega_1\) the formula \(h_i\in f_{\bar{\mathbb Z}}^{-1}[[\hat m]]\).
        \item\label{inv-poly} For each last-monic rational function \(r\in(\bar{\mathbb Q}\cap\mathbb R)(\{h_i\mid i\in\omega_1\})\), the formula \(\frac 1{r}\in f_{\bar{\mathbb Z}}^{-1}[[\hat m]]\).
        \item\label{larger} For each \(i\in\omega_1\) and \(q\in\mathbb Q\) such that \(q<g_i\) the formula \(q\le h_i\) (where \(\le\) is from Definition \ref{field-le})
        \item\label{smaller} For each \(i\in\omega_1\) and \(q\in\mathbb Q\) such that \(q>g_i\) the formula \(h_i\le q\) (where \(\le\) is again from Definition \ref{field-le})
    \end{enumerate}
    Let us show that this set is finitely satisfiable by elements of \(\mathbb F_{\hat p}\). Let \(\Phi\) be a finite subset of these formulae, and let \(h_{i_j}\) for some \(0\le j<n\) be the variables that appear in them for \(i_j\in\omega_1\). W.l.o.g these \(i\)s are sorted from smallest to largest, and so we choose values in \(\bar{\mathbb Z}\) sequentially: In every step 
    we consider all the formulae that only mention \(h_{i_k}\) for \(k\le j\). Those that do not mention \(h_{i_j}\) are already satisfied. By choosing \(v_{i_j}\in\bar{\mathbb Z}\) we satisfy the formulae of type \ref{in-set}. Since smaller values are already chosen, each formula of type \ref{inv-poly} can be satisfied by requiring that for some monic polynomial \(p\in (\bar{\mathbb Q}\cap\mathbb R)[x]\), \(\frac 1{p(x)}\in\mathbb Z\). Formulae of type \ref{larger} and \ref{smaller} only require that \(v_{i_j}\) is in a segment \(I\) has rational endpoints and contains the irrational \(g_{i_j}\) and so is has positive length. This means that by Lemma \ref{alg-integer} we can choose \(v_{i_j}\) as needed.

    From \cite[Corollary 4.4.24]{CK} \(\ast{V(\mathbb N)}\) is an \(\aleph_1\)-saturated nonstandard model, and therefore there is a mapping \(f\) that sends each \(g_i\) to some element of \(\mathbb F_{\hat p}\) such that all the formulae are satisfied by substituting \(f(g_i)\) in \(h_i\). From the formulae of type \ref{in-set} they are in fact all in \(f_{\bar{\mathbb Z}}^{-1}[[\hat m]]\subset\bar C_{\hat p}\), and from the formulae of type \ref{larger} and \ref{smaller} their real closure, which is included in \(\mathbb R^C\), is isomorphic to that of \(\mathbb Q(g_0,g_1,\dots)\) which is \(\mathbb R\). It is left to show that this real closure is a subset of \(\mathbb N^{-1}\bar C_{\hat p}\).

    This is proven by transfinite induction: For each \(\alpha\le\omega_1\) we prove that the real closure \(K_i\) of \(\mathbb Q(\{g_i\mid i<\alpha\})\) is included in \(\mathbb N^{-1}\bar C_{\hat p}\). The base case is \(\alpha=0\), where the real closure is \(\bar {\mathbb Q}\cap\mathbb R\subset\mathbb N^{-1}\bar C_{\hat p}\) since it is even more restricted. The limit case is also easy since the field there is a union of the previous cases. Let us prove the successor case. Let \(i<\omega_1\) and let \(x\in K_{i+1}\). It is an element of some algebraic extension of \(K_{i}(g_i)\), and so it is of the form \(\frac yz\) where both \(y\) and \(z\) are in some integral algebraic extensions of \(K_i[g_i]\). In particular, \(z\) has a norm \(\|z\|\in K_i[g_i]\) which satisfies that \(\|z\|x=y\cdot\frac{\|z\|}z\) is integral over \(K_i[g_i]\). Therefore by IH it is also in \(\mathbb N^{-1}\bar C_{\hat p}\), and it is left to show that \(\frac 1{\|z\|}\in \mathbb N^{-1}\bar C_{\hat p}\). Since \(\|z\|\) an element of \(K_i[g_i]\) it can be factored to a leading coefficient \(c\in K_i\) and a monic polynomial \(m\in K_i[g_i]\). \(\frac 1c\in K_i\subset\mathbb N^{-1}\bar C_{\hat p}\). As for \(m\), it is contained in some algebraic extension of \(\mathbb Q(\{g_j\mid j<i\})[g_i]\), and so has a norm \(\|m\|\) which on one hand satisfies \(\frac{\|m\|}m\in K_i[g_i]\subset\mathbb N^{-1}\bar C_{\hat p}\) and on the other hand is a last-monic function. Therefore \(\frac 1{\|m\|}\in\bar C_{\hat p}\) from a formula of type \ref{inv-poly} and so \(\frac1m=\frac{\|m\|}{m}\frac1{\|m\|}\in\mathbb N^{-1}\bar C_{\hat p}\).
\end{proof}

\section{The case where the field contains a square root of -1}\label{complex}
Let us first prove Theorems \ref{r-delta-good} and \ref{r-sigma} in this case. For that we assume that \(\mathbb F_{\hat p}\) is only known to include a copy of \(\bar{\mathbb Q}\cap\mathbb R\) and an element \(i\) that satisfies \(i^2=-1\), i.e. that it contains a copy of \((\bar{\mathbb Q}\cap\mathbb R)[i]=\bar{\mathbb Q}\).
\begin{definition}
    For any natural number \(d>1\), define \(\phi_d(\hat n)\) as \[\forall a_0,\dots,a_{d-1}\in f_{\bar{\mathbb Q}}^{-1}[[\hat n]].\exists x\in\mathbb F_{\hat p}.~x^d+\sum_{i=0}^{d-1}a_ix^i=0\]
\end{definition}
Note that similarly to before these formulae hold for all \(n\in\mathbb N\), as by proposition \ref{f-qbar}\footnote{The reader is again welcome reread the proof and convince themself that it does not rely on the existence of transcendental real numbers in \(\mathbb F_{\hat p}\).}, \(f_{\bar{\mathbb Q}}^{-1}[\mathbb N]\) is isomorphic to \(\bar{\mathbb Q}\), and hence is an algebraically closed field. However, they are again false for \(\hat p\) since \(f_{\bar{\mathbb Q}}^{-1}[[\hat p]]=\mathbb F_{\hat p}\), which is a pseudo-finite field and hence has an algebraic extension of any order. Therefore, for each \(d\) there is some maximal nonstandard natural \(\hat n_d\) that satisfies \(\phi_d\).
\begin{definition}
    \[C_{\hat p}=\bigcap_{\substack{1<d\in\mathbb N\\i\in\mathbb N}}\left[\left\lceil\sqrt[i]{\hat n_d}\right\rceil\right]\]
\end{definition}
This \(\delta\)-cut also includes \(\mathbb N\) and so from Theorem \ref{not-both} the inclusion is strict.
\begin{definition}
    \[\mathbb C_{\hat p}^{\delta}=f_{\bar{\mathbb Q}}^{-1}\left[C_{\hat p}\right]\]
    We again omit \(\hat p\) when it is clear from the context.
\end{definition}
Note that by Corollary \ref{function-delta} this is again a \(\delta\)-set.

The following implies Theorem \ref{r-delta-good} in this case:
\begin{theorem}\label{semiring-C}
    For each cut \(\mathbb N\ne C\subseteq C_{\hat p}\) that is also a semi­ring, \(\mathbb C^C\coloneq f_{\bar{\mathbb Q}}[C]\) is an algebraically closed field of cardinality at least \(\mathfrak c\) and therefore contains at least \(2^\mathfrak c\) copies of \(\mathbb R\). In particular \(\mathbb C^{\delta}\) is an algebraically closed field of cardinality at least \(\mathfrak c\).
\end{theorem}
\begin{proof}
    From propositions \ref{f-qbar-properties}\footnote{The reader is again welcome reread the proof and convince themself that it works for any field for which the function is well defined.} and \ref{subfield-alg} this is a field which is algebraically closed relative to \(\mathbb F_{\hat p}\). Let us show that it is algebraically closed. Let \[p(x)=x^d+\sum_{i=0}^{d-1}a_ix^i\] be a monic polynomial with coefficients in \(\mathbb C^C\). Since \(a_i\in\mathbb C^C\), \(f_{\bar{\mathbb Q}}(a)\in C\), and in particular \[\hat m=\max_if_{\bar{\mathbb Q}}(a_i)\in C\subseteq C_{\hat p}.\] Therefore, since \(C_{\hat p}\) is a semi­ring, \(\hat n=\hat m+1\in C_{\hat p}\), which means that \(\phi_{d}(\hat n)\) holds. Therefore there is some \(x\in\mathbb F_{\hat p}\) such that \(p(x)=0\), and since \(\mathbb C^C\) is algebraically closed relative to \(\mathbb F_{\hat p}\), \(x\in\mathbb C^C\).

    Next, let us show that \(\left|\mathbb C^C\right|\ge\mathfrak c\), from which it will follow that it contains a copy of \(\mathbb C\). Define the formula \(\psi(\hat n)\) as \(\hat n=0\lor\exists x\in f_{\bar{\mathbb Q}}^{-1}[\{\hat n\}].\) As \(f_{\bar{\mathbb Q}}(1)=1\) and \(f_{\bar{\mathbb Q}}(\sqrt[n-1]2)=n\) for all \(n\ge 2\), this formula holds for all \(n\in\mathbb N\). However it is false for \(\hat p\). Therefore there is some minimal non-standard natural \(\hat k\) that does not satisfy \(\psi\). In particular, since \(C\ne\mathbb N\), we can choose \(\hat k'\in C\) such that all the elements of \(\left[\hat k'\right]\) satisfy \(\psi\). Therefore \[\left|\mathbb C^{C}\right|=\left|f_{\bar{\mathbb Q}}^{-1}[C]\right|\ge\left|f_{\bar{\mathbb Q}}^{-1}\left[\left[\hat k'\right]\right]\right|\ge\left|\left[\hat k'\right]\right|\ge\mathfrak c.\]

    Finally, let \(\mathbb R^*\) be any \(\aleph_1\)-saturated real-closed field of cardinality \(\mathfrak c\), and note that \(\mathbb R^*[i]\cong\mathbb C\) since it is algebraically closed and has the cardinality of the continuum. Therefore the rest of the proof of Theorem \ref{r-delta-good} from Section \ref{superfields} shows that there are \(2^{\mathfrak c}\) copies of \(\mathbb R\) in \(\mathbb C\) which can be embedded in \(\mathbb C^C\).
\end{proof}

\begin{proof}[Proof of Theorem \ref{r-sigma}]
    Let \(\hat n\in C_{\hat p}\setminus\mathbb N\), and let \(C=\bigcup_{i\in\mathbb N}\left[\hat n^i\right]\). As in the other case, this is a sub-\(\sigma\)-cut of \(C_{\hat p}\) that is also a semi­ring, and so by Theorems \ref{semiring-C} and \ref{function-sigma}, \(\mathbb C^C\) is an almost internal \(\sigma\)-sub­field of \(\mathbb F_{\hat p}\) that is algebraically closed and of cardinality at least \(\mathfrak c\) which contains at least \(2^\mathfrak c\) copies of \(\mathbb R\).
\end{proof}

The results in Section \ref{more} can be similarly adapted to this case as well.

To prove Theorem \ref{r-bad} in this case note that the proof of Lemma \ref{finite-subset} only used the assumption that \(\sqrt{-1}\notin\mathbb F_{\hat p}\) to show that there is an internal relation \(G\) on \(\mathbb F_{\hat p}\) such that for each two different elements \(x, y\in\mathbb R'\), \(x>y\) (where the order follows from the isomorphism with \(\mathbb R\)) iff \(G(x,y)\), and that the proof of theorem \ref{r-bad} given that lemma did not use it at all. Therefore the following implies the theorem in this case:
\begin{lemma}
    If \(F\) is an almost internal real-closed subfield of \(\mathbb F_{\hat p}\) there is an internal relation \(G\) such that for each \(x,y\in F\), \(x\ge y\) iff \(G(x, y)\).
\end{lemma}
\begin{proof}
    Let \(f\) be an internal function and \(C\) be a cut such that \(f^{-1}[C]=F\). Let \(\phi(\hat n)\) be the formula \(\exists x,y\in f^{-1}[[\hat n+1]].~x^2=-y^2\ne 0\). Since for each \(\hat n\in C\), \(f^{-1}[[\hat n+1]]\subseteq F\), this formula does not hold for any \(\hat n\in C\), but if we consider \(\hat n=f(i)\) for \(i^2=-1\), then \(1,i\in f^{-1}[[\hat n]]\) and therefore \(\phi\) holds for this \(\hat n\). Therefore there is some maximal \(\hat n\) for which this formula fails. Define \(G(x,y)\) as \[\exists z\in f^{-1}[[\hat n+1]].~x-y=z^2.\] For each \(x,y\in F\) if \(x\ge y\) then there is some \(z\in F\) such that \(x-y=z^2\). Since \(z\in F\), \(f(z)\in C\) and therefore \(f(z)\le \hat n\), i.e. \(z\in f^{-1}[[\hat n]]\), and therefore the formula holds. On the other hand, if \(x<y\), there is \(0\ne z\in F\) such that \(y-x=z^2\), and again \(z\in f^{-1}[[\hat n]]\). Therefore, if some other \(z'\in f^{-1}[[\hat n+1]]\) satisfies \(x-y=z'^2\), \(z^2=y-x=-z'^2\), implying \(\phi(\hat n)\)---a contradiction.
\end{proof}

For Theorem \ref{R-no-sigma-delta} note that this relation is only used for a saturation argument, and the argument can also work with a conjunction of countably many internal relations. Also note the only use of the requirement that \(\sqrt{-1}\notin\mathbb F_{\hat p}\) in the proof of Theorem \ref{R-no-sigma-delta} that is not as a part of Lemma \ref{finite-subset} is again to state the order relation for a saturation argument. Therefore, similarly to before, the following implies the general case of the theorem:
\begin{lemma}
    If \(F\) is a real-closed subfield of \(\mathbb F_{\hat p}\) that is either \(\sigma\) or \(\delta\) there are relations \(G_i\) for \(i\in\mathbb N\) such that for each \(x\ne y\in F\), \(x>y\) iff \(\forall i\in\mathbb N.~G_i(x,y)\).
\end{lemma}
\begin{proof}
    Let us begin by the case that \(F\) is \(\delta\), i.e. \(F=\bigcap_{i\in\mathbb N}S_i\). Define \(G_i(x, y)\) as \[\exists z\in \bigcap_{j<i}S_j.~x-y=z^2\] (where an empty intersection is \(\mathbb F_{\hat p}\)). For each \(x\ne y\in F\), if \(x>y\) then there is some \(z\in F\) such that \(x-y=z^2\). This \(z\in S_i\) for each \(i\), and therefore \(G_i(x, y)\) holds for all \(i\). On the other hand, if \(x<y\) then if there is some \(z\) such that \(x-y=z^2\) then \(z,-z\notin F\), and therefore there is sone \(i\) such that \(z,-z\notin\bigcap_{j<i}S_j\), and therefore \(G_i(x, y)\) fails.
    
    It remains to consider the case that \(F\) is \(\sigma\), i.e. \(F=\bigcup_{i\in\mathbb N}S_i\). Define \(G_i(x,y)\) as \[\neg\exists z\in S_i.~y-x=z^2.\] For each \(x\ne y\in F\), if \(x>y\) then there is no \(z\in F\) such that \(y-x=z^2\). Therefore there is no such \(z\) in any \(S_i\), which means that \(G_i(x,y)\) holds for all \(i\). On the other hand, if \(x<y\) then there is some \(z\in F\) such that \(y-x=z^2\). Therefore there is some \(i\) such that \(z\in S_i\) and therefore \(G_i(x,y)\) fails.
\end{proof}

\appendix
\section{A more elaborate discussion of cuts}\label{cuts}
The following details some observations about cuts that were proven before Lemma \ref{finite-subset} was stated as an attempt to prove Theorem \ref{R-no-sigma-delta}.

\(\mathbb N\) is an example for a cut. Another example is \(\{\hat n\in\ast{\mathbb N}\mid\forall n\in \mathbb N.~ \hat n<\hat m-n\}\) for some nonstandard natural \(\hat m\). A simpler example is \([\hat m]\coloneq\{\hat n\in\ast{\mathbb N}\mid\hat n<\hat m\}\), but this cut is internal. In fact,
\begin{lemma}
    The internal cuts are exactly either all of \(\ast{\mathbb N}\) or \([\hat n]\) for some \(\hat n\in\ast{\mathbb N}\).
\end{lemma}
\begin{proof}
    Both \(\ast{\mathbb N}\) and \([\hat n]\) are cuts by definition and they can be defined by substituting elements of \(\ast{V(\mathbb N)}\) in first order bounded quantifier formulae and so are internal. On the other hand, the notion of being a cut can also be formulated by substituting only elements of \(\ast{V(\mathbb N)}\) in a first order bounded quantifier formula and so the fact that cuts from \(\mathcal P(\mathbb N)\) are either \(\mathbb N\) or of the form \([n]\) implies that internal cuts in \(\ast{\mathcal P(\mathbb N)}\) are either \(\ast{\mathbb N}\) or of the form \([\hat n]\).
\end{proof}
In fact, a similar characterization exists for \(\sigma\)-cuts and \(\delta\)-cuts:
\begin{lemma}
    A cut is \(\sigma\) if and only if it is either the entire \(\ast{\mathbb N}\) or of the form \(\{\hat n\in\ast{\mathbb N}\mid\exists i\in\mathbb N.~\hat n<\hat a_i\}\) for some series \(\hat a_i\) of non-standard natural numbers. Similarly, it is \(\delta\) if and only if it is either the entire \(\ast{\mathbb N}\) or of the form \(\{\hat n\in\ast{\mathbb N}\mid\forall i\in\mathbb N.~\hat n<\hat a_i\}\) for some series \(\hat a_i\) of non-standard natural numbers.
\end{lemma}
\begin{proof}
    From the previous lemma cuts of these forms are \(\sigma\) and \(\delta\) respectively, and \(\ast{\mathbb N}\) is internal and so both \(\sigma\) and \(\delta\). Let \(C=\bigcup_{i\in \mathbb N}S_i\) be a \(\sigma\)-cut that is not \(\ast{\mathbb N}\), and assume w.l.o.g that no \(S_i\) is empty. Because \(S_i\) are all in \(\ast{\mathcal P(\mathbb N)}\) and are neither empty nor \(\ast{\mathbb N}\) itself, each has a maximal element \(\hat m_i\). Set \(\hat a_i=\hat m_i+1\). Now, on one hand, for every \(\hat n\in C\) there is some \(i\) such that \(\hat n\in S_i\) and therefore \(\hat n\le\hat m_i<\hat a_i\), and on the other, for any \(\hat n\in\ast{\mathbb N}\) that satisfies \(\hat n<\hat a_i\) for some \(i\), \(\hat n\le\hat m_i\in S_i\subseteq C\) and so because \(C\) is a cut, \(\hat n\in C\). Therefore \(C=\{\hat n\in\ast{\mathbb N}\mid\exists i\in\mathbb N.~\hat n<\hat a_i\}\).
    
    Now, let \(C\ne\ast{\mathbb N}\) be a \(\delta\)-cut, and let \(\hat n\in \ast{\mathbb N}\setminus C\). Define \[C'=[\hat n+1]\setminus\{\hat n-\hat m\mid \hat m\in C\}.\] For any \(\hat k<\hat m\in C'\), \(\hat k\in[\hat n]\) as well and also \(\hat n-\hat k>\hat n-\hat m\not\in C\) and so \(\hat k\in C'\), meaning that \(C'\) is also a cut, and from Proposition \ref{internal-sigma-delta}, Corollary \ref{duality} and Theorem \ref{function-sigma} we get that \(C'\) is \(\sigma\) and therefore it is a \(\sigma\)-cut i.e. it is of the form \(\{\hat m\in\ast{\mathbb N}\mid\exists i\in\mathbb N.~\hat m<a_i\}\). Set \(b_i=\hat n-a_i+1\) and note that for a non-standard natural number \(\hat m\), the following are equivalent: \[\hat m\in C\] \[\hat m<\hat n\wedge \hat n-\hat m\not\in C'\] \[\hat m<\hat n\wedge \forall i\in\mathbb N.~\hat n-\hat m\ge a_i\] \[\forall i\in\mathbb N.~\hat m\le\hat n-a_i\] \[\forall i\in\mathbb N.~\hat m<\hat n-a_i+1=b_i\] and therefore \(C\) is of the required form.
\end{proof}
Theorem \ref{not-both} says that no cuts are both \(\sigma\) and \(\delta\).
This particularly means that \(\mathbb N\) is not \(\delta\), while the other simple cut, \(\{\hat n\in\ast{\mathbb N}\mid\forall n\in \mathbb N.~\hat n<\hat m-n\}\), is not \(\sigma\).
But there is another question: can cuts be neither \(\sigma\) nor \(\delta\)? In the rest of the appendix we prove that, under some assumptions on the filter used to create the nonstandard model, the answer is affirmative.
In particular, we assume for simplicity that the index set is \(I=\omega\).
\begin{definition}
    Given a countable partition \(P=(P_j)_{j\in\omega}\) of \(\omega\) into small sets, its \emph{partition cut} \(C_P\) is defined as follows: \(\hat n\in C_P\) if and only if there is a representative \(\hat n=[n_i]\) that is bounded separately on each of the sets \(P_j\), i.e. \[C_P=\{[n_i]\mid\forall j\in\omega.\exists b\in\mathbb N.\forall i\in P_j.~n_i\le b\}\]
\end{definition}
\begin{proposition}
    Partition cuts are cuts.
\end{proposition}
\begin{proof}
    Let \(P\) be a countable partition of \(\omega\) into small sets. For any \(\hat n\in C_P\) and any \(\hat m<\hat n\) we can choose a representative \(\hat n=[n_i]\) that is bounded on the \(P_j\)s by some sequence \(b_j\). Let us choose a representative \(\hat m=[m_i]\) and alter it on a small set such that \(m_i\le n_i\) for all \(i\). Then for all \(j\in\omega\) and \(i\in P_j\), \(m_i\le n_i\le b_j\) and so \(m\in C_P\) as well.
\end{proof}
\begin{proposition}
    Partition cuts are sub-semi­rings of \(\ast{\mathbb N}\)
\end{proposition}
\begin{proof}
    It is enough to prove that they are closed under addition and multiplication. Given a partition \(P\) and two nonstandard natural numbers \(\hat n, \hat m\in C_P\) with representatives \(\hat n=[n_i],\hat m=[m_i]\) that are bounded by \(b_j\) and \(b_j'\) respectively on \(P_j\), the natural representatives for their sum and product satisfy that for each \(j\) and \(i\in P_j\), \(n_i+m_i\le b_j+b'_j\) and \(n_im_i\le b_j b'_j\), and so \(\hat n+\hat m,\hat n\hat m\in C_P\).
\end{proof}
\begin{lemma}\label{partition-sigma}
    A partition cut is \(\sigma\) if and only if it is the entire \(\ast{\mathbb N}\).
\end{lemma}
\begin{proof}
    \(\ast{\mathbb N}\) is internal and so in particular \(\sigma\). Assume there is a partition \(P\) such that \(C_P\ne \ast{\mathbb N}\) is \(\sigma\). This means that it is of the form \(\{\hat n\in\ast{\mathbb N}\mid\exists i\in\mathbb N.~\hat n<\hat a_i\}\). W.l.o.g \(\hat a_n\ge\hat a_m\) for each \(n\ge m\), since otherwise setting \(\hat a_n=\hat a_m\) will not affect the resulting set. For each \(n\), \(\hat a_n-1\in C_P\) and so we can choose a representative of \(\hat a_n-1\), and hence also \([a^n_i]=\hat a_n\), that is bounded on the sets \(P_j\), and again w.l.o.g \(a^n_i\ge a^m_i\) for each \(n\ge m\) and for all \(i\), since otherwise the representative can be altered on a small set for that to be true. Define a series \(n_i\) like so: For each \(i\in \omega\) there is \(j\in\omega\) such that \(i\in P_j\). Define \(n_i=a^j_i\). Now, for each \(j\) we can choose \(b\) that bounds \(a^j_i\) on \(i\in P_j\), and therefore it also bounds \(n_i\) on these \(i\)s which means that \(\hat n=[n_i]\in C_P\). However, that would mean that \(\hat n<\hat a_j\) for some \(j\), but for all \(i \in\bigcup_{k\ge j} P_k\), we can choose \(k\) such that \(i\in P_k\) and observe that because \(k\ge j\), \(n_i\ge a^k_i\ge a^j_i\), and because \(\bigcup_{k\ge j} P_k\) is a large set, \(\hat n\ge\hat a_j\)---a contradiction.
\end{proof}
\begin{corollary}
    A partition cut is internal if and only if it equals \(\ast{\mathbb N}\).
\end{corollary}
\begin{proof}
    \(\ast{\mathbb N}\) is internal, and every partition cut that is not \(\ast{\mathbb N}\) is not \(\sigma\) and so in particular not internal.
\end{proof}
\begin{definition}
    An ultra­filter \(\mathcal F\) is called a \emph{P-point} (or weakly selective) if for every partition \(P\) of \(\omega\) to small sets \(\forall i.~P_i\notin\mathcal F\), there is a large set \(S\in\mathcal F\) such that \(\forall i.~|S\cap P_i|\in\omega\) is finite.
\end{definition}
\begin{definition}
    A partition \(P\) of \(\omega\) is a \emph{non-weak-selectivity witness} (or NWSW for short) of an ultra­filter \(\mathcal F\) if all of the sets \(P_i\) are small, and no set \(S\in\mathcal F\) intersects each \(P_i\) in a finite set.
\end{definition}
Note that an ultra­filter has a non-weak-selectivity witness iff it is not a P-point.
\begin{lemma}\label{P-external}
    The external partition cuts are exactly those of the NWSWs of the ultra­filter used to construct the nonstandard model. In particular, ones exist iff this ultra­filter is not a P-point.
\end{lemma}
\begin{proof}
    Let \(P\) be a NWSW of the ultra­filter 
    and let us consider the element \([id]\in\ast{\mathbb N}\) for the identity function \(\forall i\in\omega.~id_i=i\). This element cannot have a representative that is bounded on all the sets \(P_j\), since if such representative \(n_i\) is bounded by \(b_j\) on \(P_j\) then for all \(i\in \omega\) let \(j_i\) be the index such that \(i\in P_{j_i}\) and note that \[\{i\in\omega\mid i=n_i\}\subseteq \{i\in\omega\mid i<b_{j_i}\}.\] However the former is large but the latter is small, since it intersects every \(P_j\) in only a finite number of places, which is a contradiction.

    On the other hand, if \(P\) is not a NWSW then either one of the sets \(P_i\) is large, in which case the partition cut is undefined, or there is a large set \(S\) that intersects each element of \(P\) in a finite subset. Therefore for any \(\hat n\in\ast{\mathbb N}\) we can choose a representative of it and alter it such that in each set \(P_j\) it does not go over its maximum on \(S\cap P_j\) (which exists since it is a finite set). Because we changed the representative only outside \(S\) we changed it only on a small set and hence we found a representative of \(\hat n\) which is bounded on each \(P_j\). Therefore \(C_P=\ast{\mathbb N}\) is internal.
\end{proof}
\begin{definition}
    Given a filter \(\mathcal F\), a series of filters \((\mathcal D_0,\mathcal D_1,\dots)\) and a partition \(P=(P_0,P_1,\dots)\) of \(\omega\) into countable sets, the \emph{Fubini sum} \[\tilde{\mathcal F}=\sum_{\mathcal F}^P\mathcal D\] is defined as the subset of \(\mathcal P(\omega)\) containing all the sets whose intersection with a large portion of the sets \(P_i\) according to \(\mathcal F\) is large according to the corresponding filters \(\mathcal D_i\): \[\tilde {\mathcal F}\coloneq\{S\in\mathcal P(\omega)\mid \{i\in\omega\mid\phi_{P_i}^{-1}[S\cap P_i]\in\mathcal D_i\}\in \mathcal F\}\] where for countable \(T\subseteq \omega\), \(\phi_T:\omega\to T\) is the set isomorphism that maps \(0\) to the smallest element in \(T\), \(1\) to the one after it etc.
\end{definition}
From the notes after Definition 8 in \cite{Garner} it follows that this indeed produces an ultra­filter.
\begin{lemma}\label{composition}
    If \(\mathcal F\) and \(\mathcal D_i\) for every \(i\) are ultra­filters and \(P\) is a countable partition of \(\omega\) into countable sets, then \(P\) is a NWSW of \(\sum_{\mathcal F}^P\mathcal D\).
\end{lemma}
\begin{proof}
    Notate \(\tilde{\mathcal F}=\sum_{\mathcal F}^P\mathcal D\).
    If we have \(S\) such that \(S\cap P_i\) is finite for every \(i\), this means that \(\phi_{P_i}^{-1}[S\cap P_i]\) is finite as well and so it is not in \(\mathcal D_i\). Therefore the set of \(i\)s corresponding to \(S\) is empty and in particular not in \(\mathcal F\), which means that \(S\not\in\tilde{\mathcal F}\).
\end{proof}
\begin{theorem}
    If the ultra­filter used to define the nonstandard model is a Fubini sum \(\tilde{\mathcal F}=\sum_\mathcal F^P\mathcal D\), then there is a cut which is neither \(\sigma\) nor \(\delta\).
\end{theorem}
\begin{proof}
    From Lemma \ref{composition} the partition \(P\) is a NWSW of \(\tilde{\mathcal F}\) and therefore from Lemma \ref{P-external} the cut \(C_P\) is external and from Lemma \ref{partition-sigma} it is not \(\sigma\). This means that it suffices to show that it is also not \(\delta\).

    Assume that it is \(\delta\), and so because it is external it has to be of the form \(\{\hat n\in\ast{\mathbb N}\mid\forall i\in\mathbb N.~\hat n<\hat a_i\}\). W.l.o.g the sequence is decreasing, i.e. for each \(n\ge m\), \(\hat a_n\le\hat a_m\). For each \(n\in\mathbb N\), let \(a^n_i\) be a representative of \(\hat a_n\), and let \(S_n=\{j\in\omega\mid\exists b\in\mathbb N.~\phi_{P_j}^{-1}[\{i\in P_j\mid a^n_i<b\}]\in\mathcal D_j\}.\) If \(S_n\in\mathcal F\) for some \(n\) then we can alter the representative on a small set such that on the \(P_j\)s for all \(j\in S_n\) it will be bounded by the corresponding \(b\) and on the \(P_j\)s for \(j\notin S\) it will be \(0\) and therefore also bounded, which means that \(\hat a_n\in C_P\). Therefore for all \(n\), \(S_n\not\in\mathcal F,\) which means that, because the \(P_j\)s are countable, we can change the representative on a small set such that for no \(j\) there is a bound \(b\) such that \(\phi_{P_j}^{-1}[\{i\in P_j\mid a^n_i<b\}]\in\mathcal D_j\). Now, also alter the representatives to be decreasing and note that it does not change the previous condition, since for each \(j\) and \(b\), \[\phi_{P_j}^{-1}\left[\left\{i\in P_j\middle| \min_{k\le n}a^k_i<b\right\}\right]=\bigcup_{k\le n}\phi_{P_j}^{-1}\left[\left\{i\in P_j\middle| a^k_i<b\right\}\right]\not\in\mathcal D_j.\]
    Now, define a sequence \(n_i\) like so: For each \(i\), there is \(j\) such that \(i\in P_j\). Set \(n_i=a^j_i\). It satisfies \(\hat n=[n_i]\le\hat a_j\) for all \(j\), since this happens on \(\bigcup_{k\ge j}P_k\). Therefore, because the \(\hat a_j\)s are decreasing it also holds that \(\hat n<\hat a_j\) for all \(j\) and therefore \(\hat n\in C_p\). This means that there is some representative \(\hat n=[n'_i]\) that is bounded on every \(P_j\), i.e. \(\forall j.\exists b.\forall i\in P_j.~ n'_i<b\). However for all \(j\in\omega\), \(n_i\) satisfies \[\phi_{P_j}^{-1}[\{i\in P_j\mid n_i<b\}]=\phi_{P_j}^{-1}[\{i\in P_j\mid a^j_i<b\}]\not\in \mathcal D_j\] and therefore \(\forall j.~\phi^{-1}_{P_j}[\{i\in P_j\mid n_i=n'_i\}]\not\in\mathcal D_j\),  but this implies that the set \(\{i\in\omega\mid n_i= n'_i\}\not\in\tilde{\mathcal F}\), which contradicts the fact that they are representatives of the same element.
\end{proof}
\begin{conjecture}
    External partition cuts are never \(\delta\).
\end{conjecture}

\section{The proof of Lemma \ref{alg-integer}}\label{alg-integer-lemma-proof}

\begin{lemma}\label{int-times-unit}
    For any algebraic integers \(a_1,\dots a_m\) there is a homogeneous integer polynomial \(h\in\mathbb Z[x_1,\dots,x_m]\) of some degree \(n\) such that \(h(a_1,\dots,a_m)=kg\) where \(k\in\mathbb Z\) and \(g\) generates the ideal \(\langle a_1,\dots,a_m\rangle^n\).
\end{lemma}
\begin{proof}
    Consider the field \(K=\mathbb Q\left[\frac {a_1}{a_m},\dots\frac{a_{m-1}}{a_m}\right]\) and consider the fractional ideal \(\left\langle\frac{a_1}{a_m},\dots\frac{a_{m-1}}{a_m},1\right\rangle\) over \(\mathcal O_K\). It follows from \cite[Theorem I.6.3]{Neukirch} that there is a power of this ideal that is principal \(\left\langle\frac{a_1}{a_m},\dots\frac{a_{m-1}}{a_m},1\right\rangle^n=\langle b\rangle.\) Let us choose such a power \[n\ge\sum_{i=1}^m\left(\deg m_{\frac{a_i}{a_0}}-1\right)\] where \(m_\alpha\) is the minimal polynomial of \(\alpha\) and note that the set \[S=\left\{\prod_{i=1}^{m-1}\left(\frac{a_i}{a_m}\right)^{d_i}\middle|d_i<\deg m_{\frac{a_i}{a_m}}\right\}\] spans \(K\) as a linear space over \(\mathbb Q\) and therefore \(b=\sum_{\alpha\in S}c_\alpha\alpha\) for some choice of \(c_\alpha\in\mathbb Q\). We can multiply by \(a_m^n\) and get that \(g=ba_m^n\) is on one hand a generator of the ideal \(\langle a_1,\dots a_m\rangle^n\) and on the other hand a rational linear combination of elements of the form \(\alpha a_m^n\) for \(\alpha\in S\). Since all the elements of \(S\) are some products of powers of \(a_i\)s divided by a power of \(a_m\) whose exponent is the sum of their exponents, and this sum is at most \[\sum_{i=1}^m\left(\deg m_{\frac{a_i}{a_m}}-1\right)\le n,\] each element of the form \(\alpha a_m^n\) is a degree-\(n\) monomial in \(a_1,\dots,a_m\) which means that \(g\) is a homogeneous polynomial in \(a_1,\dots,a_m\) over \(\mathbb Q\). Multiply by the common denominator of its coefficients \(k\) to get a homogeneous polynomial \(h\) over \(\mathbb Z\) such that \(h(a_0,\dots,a_m)=kg\).
\end{proof}
\begin{lemma}\label{poly-nonzero}
    Let \(F\) be a finite field and let \(n>0\). Then there is a
    homogeneous integer polynomial \(f\) with \(n\) variables such that \[\forall \alpha_1,\dots,\alpha_n\in F.~f(\alpha_1,\dots,\alpha_n)=0\leftrightarrow\alpha_1=\dots=\alpha_n=0.\]
\end{lemma}
\begin{proof}
    First note that any non-constant homogeneous polynomial returns \(0\) on an input of all zeros.
    
    Let us start by proving the lemma for 2 variables. In this case, Let \(\ell\) be the characteristic of \(F\) and let us take any polynomial over \(\mathbb F_\ell\) that has no root in \(F\) and homogenize it to get a homogeneous polynomial \(f\) in two variables. \(f(x,0)\) for \(x\ne 0\) does not equal \(0\) since by being a homogenization it has a monomial that does not include the second variable, and any other nonzero point will not produce \(0\) either since then we will be able to divide the value of the first variable by that of the second and get a root of the original polynomial.

    The lemma follows for powers of 2 by composition: If \(f\) works for \(2^m\) variables and \(f_2\) is the polynomial for \(2\) variables then the following are equivalent: \[f_2(f(\alpha_1,\dots,\alpha_{2^m}),f(\alpha_{2^m+1},\dots,\alpha_{2^{m+1}}))=0\]\[f(\alpha_1,\dots,\alpha_{2^m})=f(\alpha_{2^m+1,},\dots,\alpha_{2^{m+1}})=0\]\[\alpha_1=\dots=\alpha_{2^{m+1}}=0\]
    Note that this is still a homogeneous polynomial.

    Finally, for any number of variables \(n\) we can choose \(2^m\ge n\), take the polynomial \(f\) for it and substitute zeros in all but \(n\) variables of \(f\). If the result is \(0\) all the variables of \(f\) have value \(0\) and in particular the \(n\) ones left free.
\end{proof}
\begin{lemma}\label{int-2-1}
    For an integer \(k\in\mathbb Z\) and an algebraic integer \(\alpha\), let \(K=\mathbb Q[\alpha]\). If \(\langle k,\alpha\rangle=\mathcal O_K\) there is some homogeneous polynomial \(h\in\mathbb Z[x,y]\) such that \(h(k,\alpha)=1.\)
\end{lemma}
\begin{proof}
    Let \(d=[K:\mathbb Q]\) be the degree of \(\alpha\) over \(\mathbb Q\), and for each \(m\in\mathbb N\) let us construct a vector \(v^m\in\mathbb Q^d\) such that \[\left\langle \begin{pmatrix}
        \alpha^m\\\vdots\\\alpha^{m-d+1}
    \end{pmatrix},v^m\right\rangle=\sum_{i=0}^{d-1}v^m_i\alpha^{m-i}=1.\] First, \(v^0=\begin{pmatrix}
        1\\0\\\vdots\\0
    \end{pmatrix}.\) For higher \(m\) note that if \(M\) is the companion matrix of the characteristic polynomial of \(\alpha^{-1}\) then it satisfies \(M\begin{pmatrix}
        \alpha^m\\\vdots\\\alpha^{m-d+1}
    \end{pmatrix}=\begin{pmatrix}
        a^{m-1}\\\vdots\\a^{m-d}
    \end{pmatrix}\) and therefore \[\left\langle\begin{pmatrix}
        a^m\\\vdots\\a^{m-d+1}
    \end{pmatrix},M^Tv^{m-1}\right\rangle=
    \left\langle\begin{pmatrix}
        a^{m-1}\\\vdots\\a^{m-d}
    \end{pmatrix},v^{m-1}\right\rangle=1\] which means we can choose \(v^m=M^Tv^{m-1}\). Note that since \(\langle k,\alpha\rangle=\mathcal O_K\), \(\mathcal N(\alpha)\) and \(k\) are coprime, and so since \(\mathcal N(\alpha)\alpha^{-1}\) is an algebraic integer, \(M\) (and therefore also \(M^T\)) can be defined modulo any power of \(k\), which means that the \(v^m\)s can also be defined modulo any power of \(k\). Also note that since \(\det M^T=\det M=\mathcal N(\alpha^{-1})=\mathcal N(\alpha)^{-1}\) it is invertible wherever it is defined.

    From the pigeonhole principle, there are \(m_1<m_2\) such that \(v^{m_1}\equiv_{k^{d-1}}v^{m_2}\), and because \(M^T\) is invertible modulo \(k^{d-1}\), \(v^{m_2-m_1}\equiv_{k^{d-1}}v^0\). Since each vector in the series is determined by the pervious one (and even modularly so), this is true for each multiple of \(m_2-m_1\). Let us choose some multiple \(m\ge d\) of \(m_2-m_1\) and get that \[v^m\equiv_{k^{d-1}}v^0=\begin{pmatrix}
        1\\0\\\vdots\\0
    \end{pmatrix}.\] Let \(g\) be the common denominator of the entries of \(v^m\) and note that \(g\) is coprime with \(k\), \(gv^m\) is an integer vector and \(\sum_{i=0}^{d-1}gv^m_i\alpha^{m-i}=g.\) However, since every entry of \(v^m\) but the first is \(0\) modulo \(k^{d-1}\), every summand in this sum is either an integer multiple of \(\alpha^m\) or an integer multiple of \(k^{d-1}\alpha^{m-i}\) for \(i\le d-1<m\) which is an integer multiple of \(k^i\alpha^{m-i}\). Therefore this sum is a degree-\(m\) homogeneous integer polynomial in \(k\) and \(\alpha\). \(k^m\) is also a degree-\(m\) homogeneous polynomial in \(k\) and \(\alpha\), and since \(g\) is coprime with \(k\), it is also coprime with \(k^m\), which means that \(1=ag+bk^m\) can also be expressed as a homogeneous integer polynomial in \(k\) and \(\alpha\).
\end{proof}

\begin{lemma}\label{trivial-ideal-integer-polynomial}
    For any algebraic integers \(a_1,\dots a_m\) there is a homogeneous integer polynomial \(h\in\mathbb Z[x_1,\dots,x_m]\) of some degree \(n\) such that \(h(a_1,\dots,a_m)\) is a generator of \(\langle a_1,\dots,a_m\rangle^n\).
\end{lemma}
\begin{proof}
    By Lemma \ref{int-times-unit} there is a homogeneous integer polynomial \(h_0\) of degree \(n_0\) such that \(h_0(a_0,\dots,a_m)=kg\) for \(k\in\mathbb Z\) and a generator \(g\) of \(\langle a_1,\dots,a_m\rangle^n\). Let \(b_1,\dots,b_{m'}\) be all the multiples of \(n_0\) (possibly repeating) \(a\)s, and note that \[\langle b_1,\dots,b_{m'}\rangle=\langle a_1,\dots,a_m\rangle^{n_0}\] and that \(kg\) is an integer linear combination of \(b_1,\dots,b_{m'}\), and so any homogeneous integer polynomial in \(b_1,\dots,b_{m'},kg\) is also a homogeneous integer polynomial in \(b_0,\dots,b_{m'}\) which means it is also a homogeneous integer polynomial in \(a_0,\dots,a_m\).

    Notate \(K=\mathbb Q[a_1,\dots,a_m]\). For each prime \(\ell|k\), \(\mathcal O_K/\ell\) is a product of finite fields of characteristic \(\ell\), let \(F_\ell\) be the compositum of these fields. By Lemma \ref{poly-nonzero} there is a homogeneous polynomial \(h_\ell\) for each \(\ell|k\) such that for any choice of \(\alpha_1,\dots,\alpha_{m'}\in F_\ell\), if any of them is nonzero then \(h_\ell(\alpha_1,\dots,\alpha_{m'})\ne 0\). Now, note that if \(\frac{b_1}g,\dots,\frac{b_{m'}}g\) are all mapped to \(0\) by the map to some component of \(\mathcal O_K/\ell\), then any \(x\in\langle b_1,\dots,b_{m'}\rangle\) is the product of \(g\) by an element that is mapped to \(0\) by the same map, and so \(\langle b_1,\dots,b_{m'}\rangle\ne\langle g\rangle\), which is a contradiction. Therefore, \(h_\ell\left(\frac{b_0}g,\dots,\frac{b_{m'}}g\right)\) is invertible modulo \(\ell\), i.e. \[\left\langle h_\ell\left(\frac{b_0}g,\dots,\frac{b_{m'}}g\right),\ell\right\rangle=\mathcal O_K,\] and the same is true for any power of \(\ell\). By considering powers of these polynomials that all have degree \(n_k=\lcm_{\ell|k}(\deg h_\ell)\) and using the Chinese remainder theorem we can get a homogeneous polynomial \(h_k\) of degree \(n_k\) such that \[\left\langle g^{-n_k}h_k\left(b_0,\dots,b_{m'}\right),k\right\rangle=\left\langle h_k\left(\frac{b_0}g,\dots,\frac{b_{m'}}g\right),k\right\rangle=\mathcal O_K.\]

    Let \(\alpha=g^{-n_k}h_k(b_0,\dots,b_{m'})\). By Lemma \ref{int-2-1} there is a homogeneous integer polynomial \(h_1\) of degree \(n_1\) such that \(h_1(k^{n_k},\alpha)=1\). Therefore \[h_1((kg)^{n_k},h_k(b_0,\dots,b_{m'}))=h_1(k^{n_k}g^{n_k},\alpha g^{n_k})=g^{n_1n_k}h_1(k^{n_k},\alpha)=g^{n_1n_k}\] is a degree \(n=n_0n_kn_1\) homogeneous polynomial in \(a_0,\dots,a_k\) which is a generator of \(\langle a_0,\dots,a_k\rangle^n\).
\end{proof}

We recall the basic notion of a resultant two polynomials, which can be defined either as \[\res(f,g)=a^mb^n\prod_{i=1}^n\prod_{j=1}^m(\alpha_i-\beta_j)=a^m\prod_{i=1}^ng(\alpha_i)\] for \(f=a\prod_{i=1}^n(x-\alpha_i)\) and \(g=b\prod_{j=1}^m(x-\beta_j)\), or as the determinant of the Sylvester matrix \[\res(f,g)=\begin{vmatrix}
    a_n&a_{n-1}&a_{n-2}&\dots&a_0&0&\dots&0\\
    0&a_n&a_{n-1}&\dots&a_1&a_0&\ddots&\vdots\\
    \vdots&\ddots&\ddots&\ddots&\ddots&\ddots&\ddots&0\\
    0&\dots&0&a_n&a_{n-1}&\dots&a_1&a_0\\
    b_m&b_{m-1}&b_{m-2}&\dots&b_0&0&\dots&0\\
    0&b_m&b_{m-1}&\dots&b_1&b_0&\ddots&\vdots\\
    \vdots&\ddots&\ddots&\ddots&\ddots&\ddots&\ddots&0\\
    0&\dots&0&b_m&b_{m-1}&\dots&b_1&b_0
\end{vmatrix}\] for \(f=\sum_{i=0}^na_ix^i\) and \(g=\sum_{i=0}^mb_ix^i\). The proof of equivalence can be found in \cite[Proposition IV.8.3]{Lang}.
We note that the first definition shows that the resultant is multiplicative, and symmetric up to sign change. We also note that by multiplying the Sylvester matrix with its adjugate we can see that \(\res(f,g)\in\langle f,g\rangle,\) and this allows us to naturally define resultants of homogeneous polynomials of two variables.

\begin{lemma}\label{irreducible-res-1}
    For every irreducible integer polynomial \(f\) with coprime coefficients there is another non-constant integer polynomial \(g\) such that \(\res(f, g)=1\)
\end{lemma}
\begin{proof}
    Notate by \(\alpha\) a root of \(p\) and by \(n\) the degree of \(p\) and look at \(K=\mathbb Q[\alpha]\), which satisfies \([K:\mathbb Q]=n\). Also notate by \(c\) the leading coefficient of \(p\) and note that \(c\alpha\)'s minimal polynomial has integer coefficients and so it is an algebraic integer, \(c\alpha\in\mathcal O_K\). Now consider the ideal \(\langle c, c\alpha\rangle\). Note that if we multiply this ideal by all its conjugates (in the Galois closure of \(K\)) we get an ideal that contains \(c^n\pi\) for every product \(\pi=\alpha_i\alpha_j\alpha_k\dots\) of distinct roots of \(p\), and therefore also \(c^ns_i(\alpha_1,\dots\alpha_n)\) for every symmetric polynomial in the roots of \(p\). However, these are exactly the coefficients of \(p\) times \(c^{n-1}\). Therefore, the norm of the ideal \(\langle c,c\alpha_i\rangle\) divides \(c^{n-1}\).
    
    By Lemma \ref{trivial-ideal-integer-polynomial} there is a homogeneous polynomial \(h\) of degree \(m\) such that \(h(c,c\alpha)\) generates \(\langle c,c\alpha\rangle^m\). Therefore \(\mathcal N(h(c,c\alpha))|c^{m(n-1)}\), which means that \[\res\left(p(x),h(1,x)\right)=c^{\deg h(1,x)}\mathcal N\left(h(1,\alpha)\right)~|~
    c^{m}\mathcal N\left(\frac{h(c,c\alpha)} {c^{m}}\right)=\frac{\mathcal N(h(c,c\alpha))}{c^{m(n-1)}}\] is on one hand an integer and on the other hand divides \(1\), which means it has to be \(\pm 1\)
\end{proof}

\begin{lemma}\label{monic-multiple}
    Let \(u\) be a polynomial that is monic when looking modulo some prime number \(\ell\). Then for each \(n\ge 1\) there is some polynomial \(v\) such that \((1+\ell v)u\) is monic modulo \(\ell^n\).
\end{lemma}
\begin{proof}
    The proof is by induction on \(n\). For \(n=1\) we can choose \(v=0\) since \(u\) is already monic modulo \(\ell\). Assume that when looking at \((1+\ell v)u\) modulo \(\ell^n\) it is monic, but modulo \(\ell^{n+1}\) it is not, i.e. it is of the form \((1+\ell v)u=\ell^nx^df+g\) where \(g\) is monic and \(\deg g=d\) is also the degree of \(u\) modulo \(\ell\). Divide \(x^df\) by \(u\) modulo \(\ell\) to get \(x^df\equiv_\ell au+b\) for \(\deg b<\deg_\ell u=\deg g,\) and note that \[(1+\ell v-\ell^na)u=(1+\ell v)u-\ell^nau\equiv_{\ell^{n+1}}\ell^nx^df+g-\ell^nx^df+\ell^nb=g+\ell^nb\] which is also monic.
\end{proof}

\begin{lemma}\label{coprime-lift}
    Let \(w\in\mathbb Z[x]/\ell^n\) for prime \(\ell\) and natural \(n\) be monic and let \(r_1,r_2\in\mathbb F_\ell\) be coprime monic polynomials such that \(w\equiv_\ell r_1r_2\). Then there are monic \(r_1',r_2'\in\mathbb Z[x]/\ell^n\) such that \(w=r_1'r_2'\) and \(r_i\equiv_\ell r_i'\).
\end{lemma}
\begin{proof}
    The proof is again by induction in \(n\). For \(n=1\) we can choose \(r'_i=r_i\). Assume that there are monic \(r_1'\) and \(r_2'\) such that \(w\equiv_{\ell^n}r_1'r_2'\) but not modulo \(\ell^{n+1}\), i.e. \(w=r_1'r_2'+\ell^ne\) for \(\deg e<\deg w\), and also \(r'_i\equiv_\ell r_i\). Since \(r_1\) and \(r_2\) are coprime polynomials over a field, there are \(f, g\) such that \(fr_1+gr_2=1\). Divide \(ef\) by \(r_2\) to get \(ef=ar_2+f'\), such that \(\deg f'<\deg r_2\). Similarly, \(eg=br_1+g'\). Now, \[e=efr_1+egr_2=ar_1r_2+f'r_1+br_1r_2+g'r_2=(a+b)r_1r_2+f'r_1+g'r_2.\] Since \(\deg e<\deg w=\deg(r_1r_2)\), \(\deg (f'r_1),\deg (g'r_2)<\deg(r_1r_2)\) as well, and \(r_1r_2\) is monic, \(e=f'r_1+g'r_2\), meaning that \(r''_1=r'_1+\ell^ng'\) and \(r''_2=r'_2+\ell^nf'\) are monic polynomials equivalent to \(r_1\) and \(r_2\) respectively modulo \(\ell\) that satisfy \begin{align*}
        r''_1r''_2&=(r'_1+\ell^ng')(r'_2+\ell^nf')\equiv_{\ell^{n+1}}r'_1r'_2+\ell^n(f'r'_1+g'r'_2)\equiv_{\ell^{n+1}}\\&\equiv_{\ell^{n+1}}r'_1r'_2+\ell^n(f'r_1+g'r_2)=r'_1r'_2+\ell^ne=w.
    \end{align*}
\end{proof}

\begin{lemma}
    Let \(f, g\in\mathbb Z[[x]]\) be integer power series in \(x\), and for a power series \(s\) let \(c(s)\) denote the \(\gcd\) of the coefficients of \(s\). Then \(c(fg) = c(f)c(g)\).
\end{lemma}
\begin{proof}
    Let \(f=c(f)f'\) and \(g=c(g)g'\). First note that \(fg=c(f)c(g)f'g'\) and therefore \(c(f)c(g)|c(fg)\). Now, assume that \(c(fg)\ne c(f)c(g)\) and let \(\ell\) be a prime dividing \(\frac{c(fg)}{c(f)c(g)}\). This means that \(\ell\) divides all the coefficients of \(f'g'=\frac{fg}{c(f)c(g)}\). Therefore \(f'g'\equiv_\ell0\), but both \(f'\) and \(g'\) have a coefficient \(\gcd\) of 1, and so \(f',g'\not\equiv_\ell0\). This means that we found zero divisors in \(\mathbb F_\ell[[x]]\), but this is an integral domain---a contradiction.
\end{proof}

\begin{lemma}\label{series-intersection}
    Let \(p,q\in\mathbb Z[x]\) be coprime polynomials with coefficient \(\gcd\) \(1\) and let \(f\in\mathbb Z[[x]]\). If both \(p|f\) and \(q|f\) then \(pq|f\).
\end{lemma}
\begin{proof}
    Since \(f\) and \(g\) are coprime integer polynomials, by the first definition \(\res(f,g)\ne 0\), and since \(\res(p,q)\in\langle p,q\rangle=\langle p\rangle+\langle q\rangle\), \[\res(pq)f\in(\langle p\rangle+\langle q\rangle)(\langle p\rangle\cap\langle q\rangle)\subseteq\langle p\rangle\langle q\rangle=\langle pq\rangle,\] therefore \(pq|\res(p,q)f\), i.e. there is a power series \(g\) such that \(pqg=\res(p,q)f\). However, from the previous lemma, \[c(g)=c(p)c(q)c(g)=c(pqg)=c(\res(p,q),f)=\res(p,q)c(f),\] and therefore \(\frac g{\res(p,q)}\) is also an integer power series. This means that \[pq\frac g{\res(p,q)}=\frac{\res(p,q)f}{\res(p,q)}=f,\] and therefore \(pq|f.\)
\end{proof}

\begin{definition}
    Two ideals \(I, J\) in a ring \(R\) are called \emph{comaximal} if \(I+J=R\).
\end{definition}

We note that if \(I,J\) are comaximal then \(I\cap J=(I\cap J)(I+J)\subseteq IJ\subseteq I\cap J\) and so their product and intersection are equal. Also note that if \(I\) is comaximal with both \(J\) and \(K\) then \[I+JK=I+IK+JK=I+(I+J)K=I+K=R\] and so \(I\) is also comaximal with \(JK\).

\begin{theorem}\label{ap+bq}
    For every pair of coprime homogeneous integer polynomials \(p, q\) with two variables there is 
    a finite set of ideals of the form \(\langle \ell^n,r\rangle\) for prime integer \(\ell\) and homogeneous polynomial \(r\) that is equivalent modulo \(\ell\) to the power of an irreducible polynomial modulo \(\ell\), \(r\equiv_\ell s^m\), such that every homogeneous integer polynomial \(h\) of degree 
    \(\ge \deg p+\deg q-1\)
    in the intersection of these ideals can be expressed as \(h=ap+bq\) for some homogeneous integer polynomials \(a, b\). Also, for each such ideal, \(p,q\in\langle \ell,r\rangle\)
\end{theorem}
\begin{proof}
    Let us prove first the case of non-homogeneous polynomials of one variable. Since \(\res(p, q)\in\langle p, q\rangle\), \(\langle p, q\rangle=\langle p, q,\res(p, q)\rangle\). Factor \(\res(p, q)=\prod_{i=0}^{k-1}\ell_i^{n_i}\) and note that \(\ell_i^{n_i}\) are pairwise coprime and the ideals \(\langle p, q, \ell_i^{n_i}\rangle\) are pairwise comaximal and therefore \[\langle p, q, \res(p, q)\rangle=\prod_{i=0}^{k-1}\langle p, q,\ell_i^{n_i}\rangle=\bigcap_{i=0}^{k-1}\langle p, q,\ell_i^{n_i}\rangle.\] 
    For each such \(\ell, n\), let \(u=\gcd_{\ell}(p,q)\) be the monic polynomial \(\gcd\) over \(\mathbb F_\ell\). Since \(\mathbb F_\ell[x]\) is a PID (and even euclidean), it follows that \(u=ap+bq\) for some polynomials \(a, b\in\mathbb F_\ell[x]\). Lift \(a\) and \(b\) to \(\mathbb Z[x]/\ell^n\) arbitrarily and set \(u=ap+bq\in\mathbb Z[x]/\ell^n\) accordingly. By lemma \ref{monic-multiple} there is \(v\in\mathbb Z[x]/\ell^n\) such that \((1+\ell v)u\) is monic. Notate \(w=(1+\ell v)u\) and note that \(w\equiv_{\ell^n}(1+\ell v)(ap+bq)\) and therefore \(w\in\langle p,q,\ell^n\rangle\), i.e. \(\langle p,q,\ell^n\rangle\subseteq\langle\ell^n,w\rangle\). Factor \(w\) to irreducible powers modulo \(\ell\), \(w\equiv_\ell\prod_{i=0}^{k-1}s_i^{m_i}\). By induction on Lemma \ref{coprime-lift} there are polynomials \(r_0,r_1,\dots\in\mathbb Z[x]/\ell^n\) such that \(w=\prod_{i=0}^{k-1}r_i\) and \(r_i\equiv_\ell s_i^{m_i}\). Now, note that if \(a\) and \(b\) are polynomials that are coprime modulo \(\ell\) then \(\langle\ell^n,a\rangle\) and \(\langle\ell^n,b\rangle\) are compaximal for each \(n\), since \[\langle\ell^n,a\rangle+\langle\ell^n,b\rangle=\langle\ell^n,a,b\rangle=\langle\ell^{n-1},a,b\rangle=\dots=\langle1,a,b\rangle=\langle1\rangle.\] This means that \[\langle\ell^n,a\rangle\langle\ell^n,b\rangle=\langle\ell^{2n},\ell^na,\ell^nb,ab\rangle=\langle\ell^n\rangle\langle\ell^n,a,b\rangle+\langle ab\rangle=\langle\ell^n\rangle+\langle ab\rangle=\langle\ell^n,ab\rangle,\] and therefore \[\langle\ell^n,w\rangle=\prod_{i=0}^{k-1}\langle\ell^n, r_i\rangle=\bigcap_{i=0}^{k-1}\langle\ell^n,r_i\rangle,\] which means that \(\langle p,q\rangle\) is a subset of an intersection of ideals of the form \(\langle\ell^n,r\rangle\) where \(r=s^m\) for some irreducible \(s\). I.e. if there is a polynomial \(h\) that is in all of these ideals then \(h\in\langle p,q\rangle\) and therefore \(h=ap+bq\) for some \(a,b\). Also, note that since \(u=\gcd_\ell(p,q)\) it divides both \(p\) and \(q\) modulo \(\ell\) and therefore \(p,q\in\langle\ell,u\rangle=\langle\ell,w\rangle\subseteq\langle\ell,r\rangle\).

    For the homogeneous case, the constraints that arise from \(p|_{y=1}\) and \(q_{y=1}\) can be homogenized to make equivalent constraints on \(h\), but more are needed. Specifically, for each \(\ell|\res(p,q)\)
    we also need to calculate \(\gcd_\ell(p|_{x=1},q|_{x=1})\), find \(r\equiv_\ell y^m\) for \(y^m||\gcd_\ell(p|_{x=1},q|_{x=1})\) such that \(p|_{x=1},q|_{x=1}\in\langle r,\ell^n\rangle\) for \(\ell^n||\res(p,q)\) (which can be done similarly to before) and require that \(h\in\langle\ell^n,r\rangle.\) For such an ideal, we can write \(\gcd_{\ell^n}(p|_{x=1},q|_{x=1})=r|_{x=1}v\) for some polynomial \(v\in\mathbb Z[y]/\ell^n\), and because \(r\equiv_\ell y^m||\gcd_\ell(p,q)\), the free coefficient of \(v\) is not zero modulo \(\ell\). Therefore, there is \(v^{-1}\in\mathbb Z[[y]]/\ell^n\) such that \(vv^{-1}\equiv_{\ell^n}1\). This means that for every \(h\in\langle\ell^n,r\rangle\), \(h|_{x=1}\in\langle p|_{x=1},q|_{x=1},\ell^n\rangle_{\mathbb Z[[y]]}\), and so, since these ideals for different \(\ell\)s are pairwise comaximal, each polynomial \(h\) that is in all of these ideals satisfies \[h|_{x=1}\in\langle p|_{x=1},q|_{x=1},\res(p,q)\rangle_{\mathbb Z[[y]]}=\langle p|_{x=1},q|_{x=1}\rangle_{\mathbb Z[[y]]},\] i.e. there are two power series \(f,g\in\mathbb Z[[y]]\) such that \(h=fp+gq\).
    
    Now, let us apply the non-homogeneous case on \(p|_{y=1}\), \(q|_{y=1}\) and \(h|_{y=1}\) to get that there are polynomials \(a\) and \(b\) such that \(h|_{y=1}=ap|_{y=1}+bq|_{y=1}\). Let \(d=\max(\deg a+\deg p,\deg b+\deg q).\) If \(d\le\deg h\), homogenize \(a\) as a \(\deg h-\deg p\) polynomial and \(b\) as a \(\deg h-\deg q\) polynomial and get the required result.
    
    Otherwise, let \(a'=a(y^{-1})y^{d-\deg p}\) and \(b'=b(y^{-1})y^{d-\deg q}\) and note that \(a'\) and \(b'\) are polynomials and \(y^{d-\deg h}h|_{x=1}=a'p|_{x=1}+b'q|_{x=1}.\) Therefore, \((a'-y^{d-\deg h}f)p|_{x=1}=(y^{d-\deg h}g-b')q|_{x=1}\), i.e. it is a power series divisible by both \(p|_{x=1}\) and \(q|_{x=1}\), which are coprime and have a coefficient \(\gcd\) of 1, and therefore from Lemma \ref{series-intersection}, \[p|_{x=1}q|_{x=1}|(a'-y^{d-\deg h}f)p|_{x=1}=(y^{d-\deg h}g-b')q|_{x=1},\] i.e. there is a power series \(s\) such that \(a'-y^{d-\deg h}f=q|_{x=1}s\) and similarly also \(y^{d-\deg h}g-b'=p|_{x=1}s\). Let \(s'\) be the polynomial that satisfies \(\deg s'<d-\deg h\) and \(y^{d-\deg h}|s-s'(y^{-1})y^{d-\deg h-1}\), and note that
    \begin{align*}
        y^{d-\deg h}~|&~(s-s'(y^{-1})y^{d-\deg h-1})q|_{x=1}=sq|_{x=1}-s'(y^{-1})y^{d-\deg h-1}q|_{x=1}=\\=&~a'-y^{d-\deg h}f-s'(y^{-1})y^{d-\deg h-1}q|_{x=1}\equiv_{y^{d-\deg h}}\\\equiv&_{y^{d-\deg h}}~a'-s'(y^{-1})y^{d-\deg h-1}q|_{x=1}
    \end{align*}
    which means that \[\deg(a-s'x^{d-\deg p-\deg q}q|_{y=1})\le d-\deg p-(d-\deg h)=\deg h-\deg p,\] and similarly \(\deg(b+s'x^{d-\deg p-\deg q}p|_{y=1})\le\deg h-\deg q.\) In addition, note that
    \begin{align*}
        &(a-s'x^{d-\deg p-\deg q}q|_{y=1})p|_{y=1}+(b+s'x^{d-\deg p-\deg q}p|_{y=1})q|_{y=1}=\\&=ap|_{y=1}+bq|_{y=1}+s'x^{d-\deg p-\deg q}p|_{y=1}q|_{y=1}-s'x^{d-\deg p-\deg q}p|_{y=1}q|_{y=1}=\\&=h|_{y=1}.
    \end{align*} 
    Homogenize and get the required result.
\end{proof}
\begin{corollary}
    for every pair of coprime homogeneous integer polynomials \(p,q\), if there are non-constant homogeneous \(f\) and \(g\) such that \(\res(p,f)=\pm1\) and \(\res(q, g)=\pm1\) then there is also a non-constant homogeneous \(h\) such that \(\res(pq, h)=\pm1\).
\end{corollary}
\begin{proof}
    Let us notate by \(\ell_i\), \(n_i\), \(r_i\), \(s_i\) and \(m_i\) the numbers and polynomials guaranteed to exist from the theorem. Since \(r_i\) is an irreducible power modulo \(\ell_i\), by dividing by the ideal \(\langle\ell_i^{n_i},r_i\rangle\) and setting either \(x\) or \(y\) to \(1\) we can construct a morphism from homogeneous polynomials to the finite ring \(\mathbb Z[x]/\langle\ell_i^{n_i},r_i'\rangle\), where \(r_i'\) is either \(r_i|_{y=1}=r_i(x,1)\) or \(r_i(1,x)\), which outputs \(0\) exactly for the polynomials in the ideal. Also, since \(r_i\equiv_{\ell_i} s_i^{m_i}\), which is irreducible modulo \(\ell_i\), dividing by \(\langle\ell_i,s'_i\rangle\) where \(s_i'\) is again either \(s_i(x,1)\) or \(s_i(1,x)\) gives a map from this ring to the field \(\mathbb Z[x]/\langle\ell_i,s_i'\rangle\) which equals zero exactly for the ring elements that are not invertible.

    Now, note that neither \(f\) nor \(g\) can be in any \(\langle \ell_i,s_i\rangle\), since that would mean that their resultants with \(p\) and \(q\) are divisible by \(\ell_i\) and in particular both are not 1. Therefore, they are mapped to nonzero elements of the field which means they are mapped to invertible elements of the ring. Since it is a finite ring its group of invertible elements is also finite which means that there are some powers \(u, v\) such that \(f^u\) and \(g^v\) are both mapped to \(1\). By raising by further powers we can bring them to the same large enough degree, i.e. there is \(f'\) a power of \(f\) that is mapped to \(1\) in all these rings and the same for \(g'\) a power of \(g\) with the same degree \(\ge\deg p+\deg q-1\). They still satisfy \(\res(p,f'),\res(q,g')=\pm1\) by the multiplicativity of the resultant. Because they are mapped to \(1\) by all the ring maps, their difference \(f'-g'\) is mapped to \(0\) by all the ring maps, meaning \(f'-g'\in \langle \ell_i^{n_i},r_i\rangle\) for all \(i\). Therefore, by the lemma, \(f'-g'=ap+bq\). This means that \(h=f'-ap=g'+bq\) satisfies \begin{align*}
        \res(pq, h)&=\res(p,h)\res(q,h)=\res(p,f'-ap)\res(q,g'+bq)\\&=\res(p,f')\res(q,g')=\pm1.
    \end{align*}
\end{proof}
\begin{corollary}\label{all-res-1}
    For every integer polynomial \(p\) with coprime coefficients there is a non-constant monic integer polynomial \(q\) such that \(\res(p, q)=\pm1\).
\end{corollary}
\begin{proof}
    Factor \(p\) to irreducible powers.
    Homogenize the irreducible polynomials and add \(y\) to the list. Homogenize the solutions that are guaranteed to exist from Lemma \ref{irreducible-res-1}, by the multiplicativity of the resultant these solutions also work for the irreducibles' powers. Use the last corollary to combine them (and for example \(x\) which achieves a resultant of \(-1\) with \(y\)) to a polynomial \(q\) that has resultant \(\pm 1\) with both \(p'\), the homogenization of \(p\), and \(y\). The resulting solution (or its opposite) will also be monic since this is equivalent to the requirement that the resultant with \(y\) is \(\pm1\). Note that since \(p'\) and \(q\) have a monomial that does not include \(y\), the former because it is the homogenization of \(p\) and the latter because its resultant with \(y\) is \(\pm1\ne 0\), \[\res(p,q|_{y=1})=\res(p'|_{y=1},q|_{y=1})=\res(p',q)=\pm1,\] showing the result holds for \(p\) as well. 
\end{proof}

\begin{lemma}\label{res-1-in-segment}
    For every integer polynomial \(p\) with coprime coefficients and every segment \(I\subset\mathbb R\) of positive length, there is a non-constant monic integer polynomial \(q\) with a root in \(I\) such that \(\res(p, q)=\pm1\).
\end{lemma}
\begin{proof}
    By Corollary \ref{all-res-1} there is some non-constant monic integer polynomial \(q\) such that \(\res(p,q)=\pm 1\). If \(\deg q<\deg p+2\), we can raise \(q\) to some power to make its degree \(\ge\deg p+2\). Since \(\res(p,q)\) is determined by the values of \(q\) at the roots of \(p\),
    adding \((mx+n)p\) to \(q\) does not change \(\res(p,q)\), and because of its large degree it does not change the fact that it is monic either either. Notate \(a=\inf I\) and \(b=\sup I\). If \(p(a)=0\) or \(p(b)=0\) increase \(a\) a little or decrease \(b\) a little so it does not happen.
    
    If \(p(a)\) and \(p(b)\) have opposite signs, there is an integer \(n\) (large enough or small enough) such that \(q(a)+np(a)\) is negative but \(q(b)+np(b)\) is positive. Therefore \(q+np\) has a root in \(I\).

    Otherwise, they have the same sign. Choose large enough \(m\) such that \[m(b-a)>\frac{q(a)}{p(a)}-\frac{q(b)}{p(b)}+1,\] and an integer \(n\) such that \[\frac{q(a)}{p(a)}+ma<n<\frac{q(b)}{p(b)}+mb.\] This means that \[\frac{q(a)}{p(a)}+ma-n<0<\frac{q(b)}{p(b)}+mb-n,\] and in particular they have different signs. Since \(p(a)\) and \(p(b)\) have the same sign it follows that \(q(a)+(ma-n)p(a)\) and \(q(b)+(mb-n)p(b)\) have different signs, and therefore, \(q+(mx-n)p\) has a root in \(I\).
\end{proof}

\begin{proof}[Proof of Lemma \ref{alg-integer}]
    First, note that we can assume w.l.o.g that there is only a single polynomial, since if the lemma is true for \(P(x)=\prod_{i<n}p_i(x)\), i.e. there is some \(\alpha\in\bar{\mathbb Z}\cap I\) such that \(\frac{1}{P(\alpha)}\in\bar{\mathbb Z}\), then for each \(i<n\), \[\frac1{p_i(\alpha)}=\frac{P_i(\alpha)}{p_i(\alpha)}\frac1{P_i(\alpha)}\in\bar{\mathbb Z}\] as well. Second, note that we can assume w.l.o.g that this polynomial has rational coefficients, since otherwise take a finite degree extension \(K\) of \(\mathbb Q\) that includes all the coefficients of \(p\), let \(S\) be the set of equivalence classes of automorphisms of \(\bar{\mathbb Q}\) where two automorphisms are equivalent if they behave the same on \(K\) (if \(K\) is chosen to be Galois \(S\) corresponds with and can be replaced by Galois group of \(K\) over \(\mathbb Q\)), notate \(P(x)=\prod_{\sigma\in S}p^\sigma(x)\) and note that this polynomial is rational and monic and as before, if \(\alpha\in\bar{\mathbb Z}\cap I\) satisfies \(\frac1{P(\alpha)}\in\bar{\mathbb{Z}}\) then also \[\frac1{p(\alpha)}=\frac{P(\alpha)}{p(\alpha)}\frac1{P(\alpha)}\in\bar{\mathbb Z}.\]

    Now, let \(p'=np\) where \(n\) is the common denominator of the coefficients of \(p\) and note that \(p'\) is an integer polynomial whose coefficients are coprime. Therefore by Lemma \ref{res-1-in-segment} there is a monic integer polynomial \(q\) such that \(\res(p',q)=\pm1\) and \(q\) has a root \(\alpha\in I\), and since \(q\) is a monic integer polynomial, \(\alpha\in\bar{\mathbb Z}\). Let \(q'\) be the minimal polynomial of \(\alpha\) and note that \(q'|q\) which means that \(\mathcal N(p'(\alpha))=\res(q',p')~|~\res(q,p')=\pm\res(p',q)=\pm1\) and therefore \(p'(\alpha)\) is a unit, which means that \[\frac1{p(\alpha)}=\frac1{\frac 1np'(\alpha)}=\frac n{p'(\alpha)}\in\bar{\mathbb Z}.\]
\end{proof}

\section*{Acknowledgments}
Thanks to Prof. Itay Kaplan from HUJI and Dr. Yatir Halevi from the Technion for their review and guidance. Thanks also to Elad Sayag, Lior Schain and Adi Ostrov for helpful conversations about polynomials and algebraic numbers that lead to the proof of Lemma \ref{alg-integer}.

\section*{A note about AI use} There was an effort to minimize the use of AI while writing this paper but there were two occasions when AI was used, both regarding Lemma \ref{alg-integer}, and in both cases no new ideas of the AI model entered the paper. In the first occasion chatGPT was asked to prove or disprove the lemma, and in separate instances some special cases of it. It wrote almost exclusively bullshit and everything that went out of these interactions was discarded. The proof that is presented Section \ref{alg-integer-lemma-proof} was written months later, inspired by the output of a hand-written search program and by conversations with friends. In the second occasion chatGPT was input a statement similar to that of Theorem \ref{ap+bq} and was asked if it was a known theorem, while it was specifically instructed to not try to prove it. Its answer indeed only listed known theorems that were similar to the statement it was input and ways in which they differed from it. One of these theorems was M. Noether's AF+BG theorem, whose proof, which was read from \cite{Fulton}, gave the approach for proving the theorem which lead to proof written here.

\bibliographystyle{plain}
\bibliography{references}{}
\end{document}